\documentclass{birkjour}
 \newtheorem{thm}{Theorem}[section]
 
 \newtheorem{lem}[thm]{Lemma}
 
 \theoremstyle{definition}
 \newtheorem{defn}[thm]{Definition}
 \theoremstyle{remark}

 \numberwithin{equation}{section}

\newcommand{\supp}{{\rm supp}}

\usepackage{amsfonts,amssymb}
\usepackage{amsmath}
\usepackage{graphicx}
\usepackage{cite}
\usepackage{enumerate}
\usepackage{textcomp} 
\usepackage{tikz} 
\usepackage{amsmath ,amssymb}
\usepackage{color,bbm,epsfig,float,fancyhdr}
\newcommand{\sign}[1]{\mathrm{sgn}\left(#1\right)}
\usepackage{caption}
\renewcommand{\theequation}{\arabic{section}.\arabic{equation}}

\DeclareMathOperator{\sgn}{sgn}
\allowdisplaybreaks
\numberwithin{equation}{section}

\renewcommand{\i}{\ifmmode\mathit{\mathchar"7010 }\else\char"10 \fi}
\renewcommand{\j}{\ifmmode\mathit{\mathchar"7011 }\else\char"11 \fi}
\newcommand{\R}{\mathbb{R}}
\newcommand{\N}{\mathbb{N}}
\newcommand{\Z}{\mathbb{Z}}
\newcommand{\abs}[1]{\left|#1\right|}
\newcommand{\modulo}[1]{\left|#1\right|}
\newcommand{\norm}[1]{\left\|#1\right\|}

\newcommand{\norma}[1]{{\left\|#1\right\|}}

\newcommand{\px}{\partial_x}

\newcommand{\Cc}[1]{\mathbf{C_c^{#1}}}

\newcommand{\vfi}{\varphi}

\newcommand{\Linf}{{\mathbf{L}^\infty}}
{%
	\renewcommand{\theenumi}{(\textbf{H}.\textbf{\arabic{enumi}})}
	\renewcommand{\labelenumi}{\theenumi}
	\begin{enumerate}}%
	{\end{enumerate}}

{%
	\renewcommand{\theenumi}{(\textbf{D}.\textbf{\arabic{enumi}})}
	\renewcommand{\labelenumi}{\theenumi}
	\begin{enumerate}}%
	{\end{enumerate}}

\allowdisplaybreaks
\begin{document}
\title[Non-local traffic models with space-discontinuous flux]
{Existence of entropy weak solutions for 1D non-local traffic models with space-discontinuous flux}

\author[F. A. Chiarello]{F. A. Chiarello}

\address{DISIM,\\ University of L'Aquila, \\
	Via Vetoio, 67100 \\
	L'Aquila, Italy}

\email{felisiaangela.chiarello@univaq.it}

\thanks{
	HDC and LMV  are supported by project MATH-Amsud 22-MATH-05 “NOTION: NOn-local conservaTION laws for engineering, biological and epidemiological applications: theoretical and numerical” and by INRIA Associated Team ``Efficient numerical schemes for non-local transport phenomena'' (NOLOCO; 2018--2022). Aditionally, LMV was partially supported by ANID-Chile through the project {\sc Centro de Modelamiento Matem\'atico} (AFB170001) of the PIA Program: Concurso Apoyo a Centros Cient\'ificos y Tecnol\'ogicos de Excelencia con Financiamiento Basal and  by Fondecyt project 1181511. HDC was partially supported by the National Agency for Research and Development, ANID-Chile through Scholarship Program, Doctorado Becas Chile 2021, 21210826.}
\author{H. D. Contreras}
\address{GIMNAP-Departamento de Matem\'aticas, \\
	Universidad del B\'io-B\'io, \\
	Concepci\'on, Chile.}
\email{harold.contretas1801@alumnos.ubiobio.cl}
\author{L. M. Villada}
\address{GIMNAP-Departamento de Matem\'aticas, \\
	Universidad del B\'io-B\'io, \\
	Concepci\'on, Chile.}
\email{lvillada@ubiobio.cl}

\subjclass{35L65; 65M12; 90B20}

\keywords{Conservation laws; Traffic models; Numerical scheme; Discontinuous flux; Non-local problem.}

\date{January 1, 2004}

\begin{abstract}
	We study a 1D scalar conservation law whose non-local flux has a single spatial discontinuity. This model is intended to describe traffic flow on a road with rough conditions.
	We approximate the problem through an upwind-type numerical scheme and provide compactness estimates for the sequence of approximate solutions. Then, we prove the existence and the uniqueness of entropy weak solutions. Numerical simulations corroborate the theoretical results and  the limit model as the kernel support tends to zero is numerically investigated. 
\end{abstract}

\maketitle

  \section{Introduction}
\label{sec:introduction}
We are interested in the analysis of the well-posedness and the numerical approximation of solutions of non-local conservation laws with a single spatial discontinuity in the flux
\begin{equation}
\label{eq:CL}
\begin{cases}
\partial_{t} \rho + \px f(t,x,\rho)=0, \quad (t,x)\in (0,\infty)\times \R,\\ 
\rho(0,x)=\rho_0(x), \quad x\in \R,
\end{cases}
\end{equation}
with
\begin{eqnarray*}
	&f(t,x,\rho)=H(-x)\, \,\rho\,g(\rho)\, v_l(\omega_\eta*\rho)+H(x)\, \,\rho\,g(\rho)\, v_r(\omega_\eta*\rho) ,
\end{eqnarray*}
where $H(x)$ is the Heaviside function, with which the flux $f(x,t,\rho)$ has a discontinuity at $x=0$ if the velocity functions $v_l(\rho)$ and $v_r(\rho)$ are different. The function $g$ is positive and such that $g'(\rho)\leq 0$ and $g(\rho_{\max})=0.$ We assume that the convolution term and the kernel function $\omega_\eta$ satisfies
\begin{eqnarray}\label{hyp:convol}
(\omega_\eta*\rho)(t,x)= \int_{x}^{x+\eta} \rho(t,y) \omega_\eta(y-x) \mathrm{d}y,\quad \eta>0 \\
\omega_\eta\in \mathbf{C}^2([0,\eta],\R^+), \quad \omega'_\eta\leq 0,\quad \omega_\eta(\eta)=0,
\end{eqnarray}
and the following hypothesis hold on the velocity functions \begin{eqnarray}\label{hyp_v}
&    v_s(\rho)=k_s \psi(\rho), \quad s=l,r, \quad \psi\in \mathbf{C}^2(\R), \hbox{ s.t. }\psi'\leq 0.
\end{eqnarray}

In the traffic vehicle context $\rho$ represents the density of vehicles on the roads, $\omega_\eta$ is a non-increasing kernel function whose support $\eta$ is proportional to the look-ahead distance of drivers, that are supposed to adapt their velocity with respect to the mean downstream traffic density.
The equation in (\ref{eq:CL}) is a non-local version of a generalized Lightill-Whitham-Richards traffic model \cite{LW,R,GP} with a discontinuous velocity field \cite{CR,KR}. 

Models of conservation laws with non-local flux describe several phenomena such as slow erosion of granular flow \cite{Amadori2012, ShenZhang}, synchronization \cite{AMADORI2017}, sedimentation \cite{Betancourt2011}, crowd dynamics \cite{ColomboMercier2012}, navigation processes \cite{amorim2020non} and traffic flow \cite{BlandinGoatin2016, keimer2020existence,Chiarello,ChiarelloGoatin, ChiarelloGoatinVillada2018}.
In particular, non-local traffic models describe the behaviour of drivers that adapt their velocity with respect to what happens to the cars in front of them. See \cite{Chiarello} for an overview about non-local traffic models and \cite{chiarelloFriedrichGoatinGK} for a continuous non-local model describing the behavior of drivers on two stretches of a road with different velocities and capacities.


There are many results relating to existence, uniqueness, stability and numerical approximation of weak entropy solutions of \textit{local} conservation laws with a spatially discontinuous flux \cite{CR,KR,adimurthi2005optimal, AudussePerthame2005, burger2008family, BurgerKarlsenTowers,GNPT,GimseRisebro91,GimseRisebro92, KarlsenRisebroTowers,KarlsenTowersLax,KarlsenTowers2016}. Conversely, in the non-local case, traveling waves for a traffic flow model with rough road conditions was studied in \cite{Shen2018}, with the following velocity functions
\begin{equation*}
	v_s(\rho)=k_s (1-\rho), \quad s=l,r \quad  g(\rho)=1.
\end{equation*}
But it is worth pointing out that in the latter case with $k_l>k_r$ and $g(\rho)=1$, the non-local model does not satisfy the Maximum principle, as it is showed in \cite{chiarelloCoclite}. On the contrary, model (\ref{eq:CL}) satisfies the Maximum principle and this makes it more realistic in the sense of traffic flow dynamics. 

In this sense, the aim of this paper is manifold:
\begin{itemize}
	\item we prove the well-posedness of the non-local space-discontinuous traffic model~(\ref{eq:CL}) for a general non-increasing speed function $\psi$, approximating the problem through a monotone numerical scheme and proving standard compactness estimates;
	\item we numerically study  the limit model as the support of the kernel function tends to $0^+$. 
\end{itemize}
Following \cite{KarlsenTowersLax}, we recall the following definitions of solution.

\begin{defn}
	\label{def:sol}
	We say that a function $\rho\in(\mathbf{L}^1\cap \mathbf{L}^\infty)([0,T]\times\R; [0,\rho_{\max}])$ is a weak solution of the initial value problem (\ref{eq:CL}) if
	for any test function $\vfi\in \Cc1([0,T[\times\R; \R)$ 
	\begin{equation*}
		\int_0^T\int_\R\left( \rho \partial_{t}\vfi+ f(t,x,\rho)\px\vfi\right)\mathrm{d}t \mathrm{d}x +\int_\R \rho_0(x)\vfi(0,x)\mathrm{d}x=0.
	\end{equation*}
\end{defn}

\begin{defn}\label{def:solution}
	A function $\rho\in(\mathbf{L}^1\cap \mathbf{L}^\infty)([0,T]\times\R; [0,\rho_{\max}])$ is an entropy weak solution of (\ref{eq:CL}), if for all $c\in[0,\rho_{\max}],$ and  any test function $\varphi\in\Cc1([0,T[\times\R;\R^+)$ 
	\begin{eqnarray*}
		\int_0^{T}\int_{\R} \abs{\rho-c}\varphi_t+\sign{ \rho-c }(f(t,x,\rho)-f(t,x,c))\, \varphi_x\, \mathrm{d}x\,\mathrm{d}t\\
		\hspace{2em}	-\int_0^{T}\int_{\R^\ast} \sign{\rho-c} f(t,x,c)_x \varphi\, \mathrm{d}x\,\mathrm{d}t+ \int_{\R} \abs{\rho_0(x)-c}\varphi(0,x) \mathrm{d}x \\
		\hspace{2em}+ \int_{0}^{T}  \abs{(k_r-k_l) c\,g(c) \,\psi(\rho\ast\omega_\eta )} \varphi(t,0) \mathrm{d}t\geq0.
	\end{eqnarray*}
\end{defn}

The paper is organized as follows. 
In Section \ref{sec:numericalscheme}, we introduce the numerical scheme that we use to discretize our problem. After that, in Section \ref{sec:well-posedness} we prove the existence and uniqueness of weak entropy solutions with $ \mathbf{L}^\infty$ and $\mathbf{BV}$ bounds. Finally, in Section \ref{sec:numericaltests}, we show some numerical tests illustrating the behaviour of solutions and investigating the limit model as the support of the kernel $\eta \to 0^+$. 

\section{Numerical scheme}
\label{sec:numericalscheme}
We introduce a uniform space mesh of width $\Delta x$ and a time step $\Delta t$, subject to a CFL condition, to be detailed later on. The spatial domain is discretized into uniform cells $I_j=[x_{j-1/2},x_{j+1/2})$, where $x_{j+1/2}=x_j+\Delta x/2$ are the cell interfaces, and $x_j=j\Delta x$ the cell centers, in particular $x=0$ where the flux function changes, falls at the midpoint of the cell $I_0=[x_{-1/2},x_{1/2})$. We take $\Delta x$ such that $\eta=N\Delta x$ for some $N\in\N$. Let $t^n=n\Delta t$ be the time mesh and $\lambda=\Delta t/\Delta x$. 
We aim to construct a finite volume approximate solution $\rho_{\Delta}$ such that $\rho_{\Delta}(t,x)=\rho_j^n$ for $(t,x)\in[t^n, t^{n+1}[\times[x_{j-1/2}, x_{j+1/2})$. To this end, we approximate the initial datum $\rho_0$ with the cell averages
$$\rho^0_{j}=\frac{1}{\Delta x}\int_{x_{j-1/2}}^{x_{j+1/2}}\rho_0(x)\mathrm{d} x,$$
we denote $\omega_k:=\frac{1}{\Delta x}\int_{k\Delta x}^{(k+1)\Delta x}\omega(y)\mathrm{d}y$ for $k=0,\dots,N-1$ and set the convolution term 
$$R_{j+1/2}^n:=(\omega_{\eta}\ast \rho_{\Delta})(x_{j+1/2},t^n)\approx \Delta x\sum_{k=0}^{N-1}\omega_k\rho_{j+k+1}^n.$$
In this way we can define the following finite volume scheme $\forall j\in\Z$
\begin{equation}\label{scheme}
\rho_j^{n+1}=\rho_j^{n}-\lambda\left( F(x_{j+\frac12},\rho^n_j, \rho^n_{j+1},R^n_{j+\frac12})-F(x_{j-\frac12},\rho^n_{j-1}, \rho^n_{j}, R^n_{j-\frac12}) \right),
\end{equation}
where $F$ is a {upwind-type} numerical flux motivated by the Scheme 3 introduced in \cite{burger2008family}  
\begin{equation}\label{numflux}
F(x_{j+1/2},\rho_{j}, \rho_{j+1},R_{j+1/2})=
\begin{cases}
\rho_j g(\rho_{j+1}) \,v_l(R_{j+1/2}) \hbox{ if } x_{j+1/2}<0,\\
\rho_j g(\rho_{j+1})\, v_r(R_{j+1/2}) \hbox{ if } x_{j+1/2}>0.
\end{cases}
\end{equation}

\section{Well-posedness}\label{sec:well-posedness}
{In this Section, we prove some properties of the the finite volume scheme (\ref{scheme})-(\ref{numflux}}).
\begin{lem}\label{bounds}
	Let hypotheses \eqref{hyp_v} hold. Given an initial datum such that $0\leq \rho_j^0 \leq \rho_{\max}$ for $j\in\Z$, then the finite volume scheme \eqref{scheme}-\eqref{numflux} is such that $$0\leq \rho_j^{n+1}\leq \rho_{\max}, \qquad j\in\Z,\quad n=0,1,\dots $$ 
	under the CFL condition 
	\begin{equation}\label{cfl}
	\Delta t\leq \min_{s=l,r}\left\{ \frac{\Delta x}{\rho_{\max} k_s \norm{g'}_{\Linf}\norm{\psi}_{\Linf}}, \frac{\Delta x}{k_s \norm{g}_{\Linf}\norm{\psi}_{\Linf}} \right\}.
	\end{equation}
\end{lem} 
\begin{proof}
	By induction, assume that $0\leq\rho^n_j\leq \rho_{\max}$ for all $j\in\Z$.
	Let us consider $j\neq0$ and set $v(\rho):=k_s\psi(\rho)$ for $s=l,r.$ 
	In this case, we can observe that 
	\begin{align*}
		\rho_j^{n+1}&=\rho_j^n-\lambda\left(\rho_j^n g(\rho^n_{j+1}) v(R^n_{j+1/2})-\rho_{j-1}^n g(\rho^n_{j}) v(R^n_{j-1/2})\right)\\
		&\leq \rho^n_j+\lambda \rho_{\max} \norm{g'}_{\Linf} k_s \norm{\psi}_{\Linf} (\rho_{\max}-\rho^n_j).
	\end{align*}
	
	Under the CFL condition \eqref{cfl}, we conclude $\rho^{n+1}_j \leq \rho_{\max}$ for all $j\in \Z^\ast$.\\
	For $j=0,$ we obtain
	\begin{align*}
		\rho_0^{n+1}&=\rho_0^n-\lambda\left(\rho_0^n g(\rho_1^n) v_r(R^n_{1/2})-\rho_{-1}^n g(\rho^n_0) v_l(R^n_{-1/2})\right)\\
		&\leq\rho_0^n+\lambda\left(\rho_{\max} g(\rho^n_0) v_l(R^n_{-1/2})\right)\\
		&\leq\rho_0^n+\lambda k_l \rho_{\max}\norm{g'}_{\Linf}\norm{\psi}_{\Linf} (\rho_{\max}-\rho^n_0)\leq \rho_{\max}.
	\end{align*}
	To prove the positivity $\rho^{n+1}_j\geq 0$, we observe that 
	\begin{align*}
		\rho_j^{n+1}&=\rho_j^n-\lambda\left(\rho_j^n g(\rho^n_{j+1})v(R^n_{j+1/2})-\rho_{j-1}^n g(\rho^n_{j}) v(R^n_{j-1/2})\right)\\
		&\geq \rho_j^n\left(1-\lambda g(\rho^n_{j+1}) v(R^n_{j+1/2})\right)\\
		&\geq0.
	\end{align*}
	This concludes the proof.
\end{proof}

\begin{lem}[$\mathbf{L}^1 $ norm] \label{lem:l1}
	Let hypotheses \eqref{hyp_v} hold.
	If $\rho_0\in \mathbf{L}^1(\R;\R^{+})$ then  under the CFL condition \eqref{cfl}, the approximate solution $\rho_\Delta$ constructed through the finite volume scheme \eqref{scheme}-\eqref{numflux} satisfies 
	\begin{equation}\label{L1-norm}
	\norm{\rho_{\Delta}(t,\cdot)}_{\mathbf{L}^1}=\norm{\rho_{0}}_{\mathbf{L}^1},\quad \text{for all $t>0$}.
	\end{equation}
\end{lem}
\begin{proof}
	By induction, suppose that \eqref{L1-norm} holds for $t^n=n\Delta t$. Thanks to the positivity and the conservative form of the numerical scheme \eqref{scheme} we have 
	\begin{eqnarray*}
		\norm{\rho^{n+1}}_{\mathbf{L}^1}&=&\Delta x \sum_{j\in\Z}\rho_{j}^{n+1}
		=\norm{\rho^{n}}_{\mathbf{L}^1}.
	\end{eqnarray*}
\end{proof}

We now prove the $\mathbf{L}^1$-continuity in time by following the idea introduced in \cite{KarlsenRisebroTowers}.\\
For the sake of simplicity we use the following notation throughout the proof, let us define 
\begin{equation*}
	v^n_{j+1/2}:=\begin{cases}
		v_l(R^n_{j+1/2}), \quad \hbox{if } j< {0},\\
		v_r(R^n_{j+1/2}), \quad \hbox{if } j\geq 0.
	\end{cases}
\end{equation*}

\begin{lem} \label{lem: continuity_in_time}
	Set $N_T=\lfloor T/\Delta t \rfloor.$ 
	Let $\rho_0\in \mathbf{BV}(\R;[0,\rho_{\max}])$ with $\norm{\rho_0}_{\mathbf{L}^1}<+\infty.$
	Assume that the following CFL condition holds
	\begin{equation}\label{eq:CFL_bv}
	\Delta t \leq\min_{s=l,r} \left\{\frac{\Delta x}{\rho_{\max}k_s\|\psi\|_{\Linf}\left( \|g\|_{\Linf}+ \|g'\|_{\Linf}\right)+\Delta x \rho_{\max} \omega_\eta(0) k_s\|\psi'\|_{\Linf} \|g\|_{\Linf}}\right\}
	\end{equation}
	Then, for $n=0,...,N_T-1$ 
	\begin{equation}
	\Delta x \sum_{j\in\Z}|\rho_j^{n+1}-\rho_j^{n}|\leq \text{C(T)},
	\end{equation}
	where
	\begin{eqnarray*}
		\text{C(T)}&=&e^{({2 T \rho_{\max} \|g\|_{\Linf} \|v'\|_{\Linf}})}\Delta t\bigg(3  \max_{s=l,r}\left\{\left( \norm{\psi}_{\Linf} \norm{g}_{\Linf}+\rho_{\max}\norm{\psi'}_{\Linf}\norm{\omega}_{\mathbf{L}^1}\norm{g}_{\Linf}\right.\right.\\
		&&\left.\left.+\rho_{\max}  \norm{\psi}_{\infty} \norm{g'}\right)k_s \right\}\text{TV}(\rho_0) 
		+\Delta t\rho_{\max}|k_r-k_l|\norm{\psi}_{\Linf} \norm{g}_{\Linf} \bigg) .
	\end{eqnarray*}
\end{lem}
\begin{proof}
	{First,} we fix $j\in\Z,$ by (\ref{scheme}) we have
	\begin{eqnarray*}
		\rho_j^{n+1}-\rho_j^{n}=\rho_j^{n}-\rho_j^{n-1}-\lambda \bigg(\rho^n_j g(\rho^n_{j+1}) v^n_{j+1/2}\pm\rho^{n-1}_{j} g(\rho^{n-1}_{j+1})v^{n}_{j+1/2}-\rho^{n-1}_jg(\rho^{n-1}_{j+1})v^{n-1}_{j+1/2}\\
		-\rho^n_{j-1} g(\rho^n_{j})v^n_{j-1/2}\pm \rho^{n-1}_{j-1} g(\rho^{n-1}_{j})v^{n}_{j-1/2 }+\rho^{n-1}_{j-1} g(\rho^{n-1}_{j})v^{n-1}_{j-1/2}\bigg) \\
		\qquad \qquad=\rho_j^{n}-\rho_j^{n-1}-\lambda \left((\rho^{n}_jg(\rho^n_{j+1})-\rho^{n-1}_jg(\rho^{n-1}_{j+1})) v^n_{j+\frac12}+\rho^{n-1}_jg(\rho^{n-1}_{j+1})(v^n_{j+\frac12}-v^{n-1}_{j+\frac12})\right.\\
		- \left.(\rho^{n}_{j-1}g(\rho^n_{j})-\rho^{n-1}_{j-1}g(\rho^{n-1}_{j})) v^n_{j-\frac12}-\rho^{n-1}_{j-1}g(\rho^{n-1}_{j})(v^n_{j-\frac12}-v^{n-1}_{j-\frac12}) \right), 
	\end{eqnarray*}
	{and using the mean-value Theorem, we take  $R_{j+1/2}^{n-1/2}\in(R_{j+1/2}^{n},R_{j+1/2}^{n-1})$ such that} 
	\begin{eqnarray*}
		v_{j+1/2}^n-v_{j+1/2}^{n-1}&=&v'(R_{j+1/2}^{n-1/2})(R_{j+1/2}^n-R_{j+1/2}^{n-1})\\
		&=&v'(R_{j+1/2}^{n-1/2})\Delta x \sum_{k=0}^{N-1}\omega_k(\rho_{j+k {+1}}^n-\rho_{j+k {+1}}^{n-1}).
	\end{eqnarray*}
	Next, we can  write
	\begin{eqnarray*}
		&&\rho_j^{n+1}-\rho_j^{n}\\
		&&= \left(1-\lambda \left(v^n_{j+\frac12} g(\rho^{n-1}_{j+1})-\rho^n_{j-1}g'(\xi^{n-\frac12}_{j-\frac12})v^n_{j-\frac12}-\Delta x \omega_{ {0}}\rho_{j-1}^{n-1} g(\rho^{n-1}_{j}) v'(R_{j-\frac12}^{n-\frac12})\right)\right)(\rho_j^{n}-\rho_j^{n-1})\\
		&&\quad-\lambda \left(\rho^n_j g'(\xi^{n-1/2}_{j+1/2}) v^n_{j+1/2}-v'\left(R^{n-1/2}_{j+1/2}\right)\rho^{n-1}_jg(\rho^{n-1}_{j+1})\Delta x \omega_{ {0}}\right) (\rho^n_{j+1}-\rho^{n-1}_{j+1})\\
		&&\quad-\lambda v'(R_{j+1/2}^{n-1/2}) \rho_j^{n-1} g(\rho^{n-1}_{j+1})\Delta x \sum_{k= {1}}^{N-1}\omega_k(\rho_{j+k}^n-\rho_{j+k}^{n-1})+\lambda v_{j-1/2}^n g(\rho^{n-1}_j)(\rho^{n}_{j-1}-\rho^{n-1}_{j-1})\\
		&&\quad+\lambda\rho_{j-1}^{n-1} g(\rho^{n-1}_j)v'(R_{j-1/2}^{n-1/2})\Delta x \sum_{k= {1}}^{N-1}\omega_k(\rho_{j+k}^n-\rho_{j+k}^{n-1}).
	\end{eqnarray*}
	and thanks to the CFL condition (\ref{eq:CFL_bv}), we have
	$$1-\lambda \left(v^n_{j+1/2} g(\rho^{n-1}_{j+1})-\rho^n_{j-1} g'(\xi^{n-1/2}_{j-1/2}) v^n_{j-1/2}-\Delta x \omega_{ {0}}\rho_{j-1}^{n-1} g(\rho^{n-1}_{j}) v'(R_{j-1/2}^{n-1/2})\right)\geq0.$$ 
	Then, taking the absolute value we obtain
	\begin{eqnarray*}
		&&|\rho_j^{n+1}-\rho_j^{n}|\\
		&&\leq \left(1-\lambda \left(v^n_{j+\frac12} g(\rho^{n-1}_{j+1})-\rho^n_{j-1} g'(\xi^{n-\frac12}_{j-\frac12}) v^n_{j-\frac12}-\Delta x \omega_{ {0}}\rho_{j-1}^{n-1} g(\rho^{n-1}_{j}) v'(R_{j-\frac12}^{n-\frac12})\right)\right)|\rho_j^{n}-\rho_j^{n-1}|\\
		&&\quad+\lambda \left(-\rho^n_j g'(\xi^{n-\frac12}_{j+\frac12}) v^n_{j+\frac12}-v'\left(R^{n-\frac12}_{j+\frac12}\right)\rho^{n-1}_jg(\rho^{n-1}_{j+1})\Delta x \omega_{ {0}}\right) |\rho^n_{j+1}-\rho^{n-1}_{j+1}|\\
		&&\quad-\lambda v'(R_{j+\frac12}^{n-\frac12}) \rho_j^{n-1} g(\rho^{n-1}_{j+1})\Delta x \sum_{k= {1}}^{N-1}\omega_k |\rho_{j+k}^n-\rho_{j+k}^{n-1}|+\lambda v_{j-\frac12}^n g(\rho^{n-1}_j)|\rho^{n}_{j-1}-\rho^{n-1}_{j-1}|\\
		&&\quad-\lambda\rho_{j-1}^{n-1} g(\rho^{n-1}_j)v'(R_{j-\frac12}^{n-\frac12})\Delta x \sum_{k= {1}}^{N-1}\omega_k |\rho_{j+k}^n-\rho_{j+k}^{n-1}|.
	\end{eqnarray*}
	Now, multiplying by $\Delta x$ and summing over $j,$ we get 
	\begin{eqnarray*}
		\sum_{j\in\Z}\Delta x|\rho_j^{n+1}-\rho_j^{n}|&\leq& \sum_{j\in\Z}\Delta x|\rho_j^{n}-\rho_j^{n-1}| \\
		&&-2\lambda \sum_{j\in\Z}\Delta x\rho_{j-1}^{n-1} g(\rho^{n-1}_j)v'(R_{j-1/2}^{n-1/2})\Delta x \sum_{k= {1}}^{N-1}\omega_k|\rho_{j+k}^n-\rho_{j+k}^{n-1}|\\
		&\leq &\sum_{j\in\Z}\Delta x|\rho_j^{n}-\rho_j^{n-1}| \\
		&&+2\Delta t \rho_{\max} \|g\|_{\Linf} \|v'\|_{\Linf}  \sum_{k= {1}}^{N-1}\omega_k  \sum_{j\in\Z}\Delta x |\rho_{j+k}^n-\rho_{j+k}^{n-1}|\\
		&\leq& \left(1+2\Delta t \rho_{\max} \|g\|_{\Linf} \|v'\|_{\Linf} \sum_{k=0}^{N-1}\omega_k\right) \sum_{j\in\Z}\Delta x |\rho_{j}^n-\rho_{j}^{n-1}|.
	\end{eqnarray*}
	Thus,
	$$\sum_{j\in\Z}\Delta x|\rho_j^{n+1}-\rho_j^{n}|\leq e^{({2n\Delta t \rho_{\max} \|g\|_{\Linf} \|v'\|_{\Linf}})}\sum_{j\in\Z}\Delta x|\rho_j^{1}-\rho_j^{0}|.$$
	On the other hand, 
	\begin{eqnarray*}
		\sum_{j\in\Z}\Delta x|\rho_j^{1}-\rho_j^{0}|&\leq &\Delta t\sum_{j<0}\left|\rho_j^{0} g(\rho_{j+1}^{0}) v_l(R^{0}_{j+1/2})-\rho_{j-1}^{0} g(\rho_{j}^{0}) v_l(R^{0}_{j-1/2})\right|\\
		&&+\Delta t
		\left|\rho_{-1}^{0} g(\rho_{0}^{0}) v_l(R^0_{-1/2})-\rho_{0}^{0} g(\rho_{1}^{0}) v_r(R^0_{1/2})\right|\\
		&&+\Delta t\sum_{j>0}\left|\rho_j^{0} g(\rho_{j+1}^{0}) v_r(R^0_{j+1/2})-\rho_{j-1}^{0} g(\rho_{j}^{0}) v_r(R^0_{j-1/2})\right|.
	\end{eqnarray*}
	The first term of the right-hand side can be estimated as 
	\begin{eqnarray*}
		&&\sum_{j<0}|\rho_j^{0} g(\rho^0_{j+1})v_l(R^0_{j+1/2})-\rho_{j-1}^{0} g(\rho^0_{j})v_l(R^0_{j-1/2})| \\ 
		&&\leq \norm{v_l}_{\Linf} \norm{g}_{\Linf}\sum_{j<0}|\rho_j^{0}-\rho_{j-1}^{0}|+\rho_{\max} \norm{g'}_{\Linf} \norm{v_l}_{\Linf} \sum_{j<0}|\rho^0_{j+1}-\rho^0_j|\\
		&&\quad +\rho_{\max} \norm{g}_{\Linf}\norm{v'_l}_{\Linf}\sum_{j<0}\modulo{R^n_{j+1/2}-R^n_{j-1/2}}\\
		&&\leq \norm{v_l}_{\Linf} \norm{g}_{\Linf}\sum_{j<0}|\rho_j^{0}-\rho_{j-1}^{0}|+\rho_{\max} \norm{g'}_{\Linf} \norm{v_l}_{\Linf} \sum_{j<0}|\rho^0_{j+1}-\rho^0_j|\\
		&&\quad+\rho_{\max}\norm{g}_{\Linf}
		\norm{v'_l}_{\Linf}\sum_{j<0}\sum_{k=0}^{N-1}\Delta x\omega_k\modulo{\rho_{j+k+1}^0-\rho_{j+k}^0}\\
		&&\leq\left(\norm{v_l}_{\Linf}\norm{g}_{\Linf}+\rho_{\max}\norm{g}_{\Linf}\norm{v'_l}_{\Linf}\norm{\omega}_{\mathbf{L}^1}+\rho_{\max} \norm{g'}_{\Linf}\norm{v_l}_{\Linf}\right) \hbox{TV}(\rho_0).
	\end{eqnarray*}
	Analogously,
	\begin{eqnarray*}
		&&\sum_{j>0}|\rho_j^{0} g(\rho_{j+1}^0)v_r(R^0_{j+1/2})-\rho_{j-1}^{0} g(\rho_{j}^0)v_r(R^0_{j-1/2})|\\
		&&\leq(\norm{v_r}_{\Linf} \norm{g}_{\Linf}+\rho_{\max}\norm{v'_r}_{\Linf} \norm{g}_{\Linf}\norm{\omega}_{\mathbf{L}^1}+\rho_{\max} \norm{g'}_{\Linf}\norm{v_r}_{\Linf})\hbox{TV}(\rho_0) ,
	\end{eqnarray*}
	and by hypothesis (\ref{hyp_v})
	\begin{eqnarray*}
		&&|\rho_{-1}^{0} g(\rho^0_0)v_l(R^0_{-1/2})-\rho_{0}^{0} g(\rho^0_1)v_r(R^0_{1/2})|\\
		&&\quad\leq  |\rho_{-1}^{0}g(\rho^0_0)v_l(R^0_{-1/2})\pm \rho_{-1}^{0}g(\rho^0_0)v_r(R^0_{1/2})-\rho_{0}^{0} g(\rho^0_1)v_r(R^0_{1/2})|\\
		&&\quad\leq \rho_{\max} \norm{g}_{\Linf} \modulo{v_l(R^0_{-1/2})-v_r(R^0_{1/2})} + v_r(R^0_{1/2}) \modulo{\rho_{-1}^{0}g(\rho^0_0)-\rho_{0}^{0}g(\rho^0_1)}\\
		&&\quad\leq \rho_{\max} \norm{g}_{\Linf} \modulo{v_l(R^0_{-1/2})\pm v_r(R^0_{-1/2}) -v_r(R^0_{1/2})}  + v_r(R^0_{1/2}) \modulo{\rho_{-1}^{0}g(\rho^0_0)-\rho_{0}^{0}g(\rho^0_1)}\\
		&&\quad\leq \rho_{\max} \norm{g}_{\Linf} \norm{\psi}_{\Linf} \modulo{k_l- k_r}+\rho_{\max} \norm{g}_{\Linf} \norm{v'_r} \norm{\omega}_{\mathbf{L}^1}\hbox{TV}(\rho_0) \\
		&&\quad\quad + \left(\norm{v_r}_{\Linf} \norm{g}_{\Linf}+\rho_{\max}\norm{v_r}_{\Linf} \norm{g'}_{\Linf} \right)\hbox{TV}(\rho_0).
	\end{eqnarray*}
	Finally,
	\begin{eqnarray*}
		\sum_{j\in\Z}\Delta x|\rho_j^{1}-\rho_j^{0}|&\leq& 3 \Delta t  \max_{s=l,r} \Bigl\{  \Bigl( \norm{\psi}_{\Linf} \norm{g}_{\Linf}+\rho_{\max}\norm{\psi'}_{\Linf}\norm{\omega}_{\mathbf{L}^1}\norm{g}_{\Linf} \Bigr. \Bigr.\\
		&&\Bigl. \Bigl. +\rho_{\max}  \norm{\psi}_{\infty} \norm{g'}\Bigr) k_s \Bigr\}\hbox{TV}(\rho_0)+\Delta t\rho_{\max}|k_r-k_l|\norm{\psi}_{\Linf} \norm{g}_{\Linf}.
	\end{eqnarray*}
	This completes the proof.
\end{proof}

\subsection{Spatial BV estimates.}
\begin{lem} \label{lem:spatialbv}
	Let $\rho_0\in \Linf \cap \mathbf{BV}(\R;[0,\rho_{\max}])$. Assume that the CFL condition \eqref{cfl} holds. For any interval $[a,b]\subset \R$ such that $0 \notin [a,b],$ fix $q>0$ such that $2q<\min \{\modulo{a},\modulo{b}\}$ and $q>\Delta x.$ Then, for any $n=1,...,N_T-1$ the following estimate holds:
	\begin{eqnarray*}
		\sum_{j\in\mathbf{J}_a^b}\modulo{\rho^{n}_{j+1}-\rho^{n}_{j}}\leq e^{2\mathcal{K}T}\left( \hbox{TV}(\rho_0)+2\frac{C(T)}{q} 
		+\mathcal{K}_2T\right),
	\end{eqnarray*}
	with $\mathbf{J}_a^b=\{j\in\Z:a\leq x_j\leq b\}.$
\end{lem}
\begin{proof}
	Let 
	\begin{eqnarray*}
		\quad  \mathcal{M}_\Delta=\left\{j\in\Z: x_{j-1/2}\in[a-q-\Delta x,a]\right\},\, \mathcal{N}_\Delta=\left\{j\in\Z: x_{j+1/2}\in[b,b+q+\Delta x]\right\}.
	\end{eqnarray*}
	By the assumptions on $q,$ observe that there are at least 2 elements in each of the sets above, i.e. $\modulo{\mathcal{M}_\Delta},\,\modulo{\mathcal{N}_\Delta}\geq 2.$ Moreover, $\modulo{\mathcal{M}_\Delta}\Delta x\geq q$ and $\modulo{\mathcal{N}_\Delta}\Delta x\geq q.$ 
	By Lemma \ref{lem: continuity_in_time} there exists a constant $C(T)$ such that 
	\begin{equation}
	\Delta x \sum_{n=0}^{N_T-1}\sum_{j\in\Z}|\rho_j^{n+1}-\rho_j^{n}|\leq TC(T),
	\end{equation}
	with $C(T)$ as in Lemma \ref{lem: continuity_in_time}, then when restricting the sum over $j$ in the set $\mathcal{M}_\Delta,$ respectively $\mathcal{N}_\Delta,$ it follows that 
	\begin{equation}
	\Delta x \sum_{n=0}^{N_T-1}\sum_{j\in \mathcal{M}_\Delta}|\rho_j^{n+1}-\rho_j^{n}|\leq  TC(T) \quad\hbox{and}\quad \Delta x \sum_{n=0}^{N_T-1}\sum_{j\in \mathcal{N}_\Delta}|\rho_j^{n+1}-\rho_j^{n}|\leq  TC(T).
	\end{equation}
	Let us choose $j_a\in \mathcal{M}_\Delta$ and $j_b$ with $j_b+1\in \mathcal{N}_\Delta$ such that 
	\begin{eqnarray*}
		\sum_{n=0}^{N_T-1}|\rho_{j_a}^{n+1}-\rho_{j_a}^{n}|&=& \min_{j\in \mathcal{M}_\Delta}\sum_{n=0}^{N_T-1}|\rho_j^{n+1}-\rho_j^{n}|,\\
		\sum_{n=0}^{N_T-1}|\rho_{j_{b+1}}^{n+1}-\rho_{j_{b+1}}^{n}|&=& \min_{j\in \mathcal{N}_\Delta}\sum_{n=0}^{N_T-1}|\rho_j^{n+1}-\rho_j^{n}|, 
	\end{eqnarray*}
	thus, 
	\begin{eqnarray*}
		\sum_{n=0}^{N_T-1}|\rho_{j_a}^{n+1}-\rho_{j_a}^{n}|&\leq& \frac{C}{\modulo{\mathcal{M}_\Delta}\Delta x}\leq \frac{TC(T)}{q},\\
		\sum_{n=0}^{N_T-1}|\rho_{j_{b+1}}^{n+1}-\rho_{j_{b+1}}^{n}|&\leq& \frac{C}{\modulo{\mathcal{N}_\Delta}\Delta x}\leq \frac{TC(T)}{q}, 
	\end{eqnarray*}
	moreover, observe that
	\begin{equation}\label{eq:tv_sum}
	\sum_{j=j_a}^{j_b}\modulo{\rho^{n+1}_{j+1}-\rho^{n+1}_{j}}=\modulo{\rho^{n+1}_{j_a+1}-\rho^{n+1}_{j_a}}+\sum_{j=j_a+1}^{j_b-1}\modulo{\rho^{n+1}_{j+1}-\rho^{n+1}_{j}}+\modulo{\rho^{n+1}_{j_b+1}-\rho^{n+1}_{j_b}}.
	\end{equation}
	Now, let us focus on the central sum on the right-hand side of (\ref{eq:tv_sum}), we write 
	$$\rho_{j+1}^{n+1}-\rho_{j}^{n+1}=\mathcal{A}^n_j-\lambda \mathcal{B}^n_j,$$
	with 
	\begin{eqnarray*}
		\mathcal{A}^n_j&:=&(1-\lambda v^n_{j+3/2} g(\rho^n_{j+2})+\lambda \rho^n_{j-1} g'(\xi^n_{j+1/2}) v^n_{j+1/2})(\rho_{j+1}^n-\rho_j^n)\\
		&&+\lambda v^n_{j+1/2} g(\rho^n_{j+1})(\rho_{j}^n-\rho_{j-1}^n)-\lambda \rho^n_{j+1} g'(\xi^n_{j+3/2}) v^n_{j+3/2} (\rho^n_{j+2}-\rho^n_{j+1}) ,\\
		\mathcal{B}^n_j&:=&\rho_{j}^n g(\rho^n_{j+1})(v^n_{j+3/2}-v^n_{j+1/2})- \rho_{j-1}^n g(\rho^n_{j})(v^n_{j+1/2}-v^n_{j-1/2}).
	\end{eqnarray*}
	Taking the absolute value and summing, 
	\begin{eqnarray*}
		\sum_{j=j_a+1}^{j_b-1}|\mathcal{A}^n_j|&\leq& \sum_{j=j_a+1}^{j_b-1}|\rho_{j+1}^n-\rho_j^n|+\lambda v^n_{j_a+3/2} g(\rho^n_{j_a+2})|\rho_{j_a+1}^n-\rho_{j_a}^n|  \\
		&&-\lambda v^n_{j_b+1/2} g(\rho^n_{j_b+1}) |\rho_{j_b}^n-\rho_{j_b-1}^n|+\lambda g'(\xi^n_{j_a+3/2}) \rho^n_{j_a+1}v^n_{j_a+3/2} \modulo{\rho^n_{j_a+2}-\rho^n_{j_a+1}}\\
		&&-\lambda g'(\xi^n_{j_b+1/2}) \rho^n_{j_b-1}v^n_{j_b+1/2} \modulo{\rho^n_{j_b+1}-\rho^n_{j_b}}.
	\end{eqnarray*}
	On the other hand, 
	\begin{eqnarray*}
		\quad \mathcal{B}^n_j&=& \rho^n_j g(\rho^n_{j+1}) (v^n_{j+3/2}-v^n_{j+1/2})-\rho^n_{j-1} g(\rho^n_{j})(v^n_{j+1/2}-v^n_{j-1/2})\\
		\quad  &=&\rho^n_j g(\rho^n_{j+1}) v'(\tilde R^n_{j+1}) \left(R^n_{j+3/2}-R^n_{j+1/2}\right)-\rho^n_{j-1} g(\rho^n_{j}) v'(\tilde R^n_{j}) \left(R^n_{j+1/2}-R^n_{j-1/2}\right)\\
		&=&\rho^n_j g(\rho^n_{j+1}) v'(\tilde R^n_{j+1}) \left(R^n_{j+3/2}-R^n_{j+1/2}\right)\pm \rho^n_{j-1} g(\rho^n_{j})v'(\tilde R^n_{j+1})\left(R^n_{j+3/2}-R^n_{j+1/2}\right)\\
		&&-\rho^n_{j-1} g(\rho^n_{j}) v'(\tilde R^n_{j}) \left(R^n_{j+1/2}-R^n_{j-1/2}\right)\\
		&=&\rho^n_j g'(\xi^n_{j+1/2})(\rho^n_{j+1}-\rho^n_j) v'(\tilde R^n_{j+1})\left(R^n_{j+3/2}-R^n_{j+1/2}\right)\\
		&&+g(\rho^n_j)(\rho^n_j-\rho^n_{j-1}) v'(\tilde R^n_{j+1})\left(R^n_{j+3/2}-R^n_{j+1/2}\right)\\
		&&+ \rho^n_{j-1} g(\rho^n_{j}) v'(\tilde R^n_{j+1})\left(R^n_{j+3/2}-2 R^n_{j+1/2}+R^n_{j-1/2}\right)\\
		&&+\rho^n_{j-1} g(\rho^n_{j}) v''(\bar R^n_{j+1/2})\left(R^n_{j+1/2}-R^n_{j-1/2}\right)\left(\tilde R^n_{j+1}-\tilde R^n_{j}\right),
	\end{eqnarray*}
	where $\tilde{R}^n_j\in \mathcal{I}(R^n_{j-1/2},R^n_{j+1/2})$ and $\bar{R}^n_{j+1/2}\in \mathcal{I}(\tilde R^n_{j},\tilde R^n_{j+1}).$ Now, by the assumptions (\ref{hyp_v}) on the kernel function and defining $\omega_N:=0$, we get
	\begin{eqnarray*}
		\modulo{R^n_{j+1/2}-R^n_{j-1/2}}&=&\modulo{\Delta x\sum_{k=0}^{N-1}\omega_k (\rho_{j+k+1}^n-\rho_{j+k}^n)} \\
		&=&\Delta x \left|-\omega_0 \rho_j^n +\sum_{k=1}^{N-1}(\omega_{k-1}-\omega_{k})\rho_{j+k}^n +\omega_{N-1} \rho_{j+N}^n\right|\\
		&\leq& \Delta x\left( {2}\omega_\eta(0) \rho_{\max}+ \norm{\omega'}_{\mathbf{L}^\infty}\norm{\rho}_{\mathbf{L}^1}\right),
	\end{eqnarray*}
	and 
	\begin{eqnarray*}
		|R^n_{j+3/2}-2 R^n_{j+1/2}+R^n_{j-1/2}|&=&\left|\Delta x\left( \sum_{k=0}^{N-1}\omega_k \rho_{j+k+2}^n-2\sum_{k=0}^{N-1}\omega_k \rho_{j+k+1}^n+\sum_{k=0}^{N-1}\omega_k \rho_{j+k}^n\right)\right|\\
		&=&\left|\Delta x\left(\sum_{k=1}^{N-1}(\omega_{k-1}-2\omega_{k}+\omega_{k+1})\rho_{j+k+1}^n+\Delta x \rho^n_{j+1}\frac{\omega_1-\omega_0}{\Delta x}\right.\right.\\
		&&+\left.\left. \frac{1}{\Delta x}(\omega_{N-1}\underbrace{-\omega_N }_{:=0})\rho^n_{j+N+1} \Delta x +\omega_0(\rho^n_{j}-\rho^n_{j+1})\right)\right| \\
		&\leq& \left(\Delta x\right)^2\norm{\omega''}_{\Linf}\norm{\rho}_{\mathbf{L}^1}+2(\Delta x)^2 \rho_{\max} \norm{\omega'}_{\Linf}\\
		&&+\Delta x \omega_0 |\rho_{j}-\rho_{j+1}|.
	\end{eqnarray*}
	Now, we compute,
	$\modulo{\tilde R^n_{j+1}-\tilde R^n_{j}},$
	\begin{eqnarray*}
		\modulo{\tilde{R}^n_{j+1}-\tilde{R}^n_j}&=&\modulo{\theta R^n_{j+3/2}+(1-\theta)R^n_{j+1/2}-\mu R^n_{j+1/2}-(1-\mu)R^n_{j-1/2}} \\
		&\leq& 3 \Delta x \left( {2\omega_{\eta}(0)\rho_{\max}}+\|\omega'\|_{\Linf}\|\rho\|_{\mathbf{L}^1}\right),
	\end{eqnarray*}
	for some $\theta, \mu\in[0,1].$
	We end up with 
	\begin{eqnarray*}
		\modulo{\mathcal{B}^n_j}&\leq& \Delta x ( {2}\omega_\eta(0)\rho_{\max}+\|\omega'\|_{\Linf}\|\rho\|_{\mathbf{L}^1}) \norm{g}_{\Linf} \norm{v'}_{\Linf} \modulo{\rho^n_j-\rho^n_{j-1}}\\
		&&+  (\Delta x)^2 \norm{g}_{\Linf}\left(\norm{\omega''}_{\Linf} \norm{\rho}_{\mathbf{L}^1}\norm{v'}_{\Linf}+2  \rho_{\max} \norm{\omega'}_{\mathbf{L}^\infty}\norm{v'}_{\Linf} \right) \modulo{\rho^n_{j-1}}\\
		&&+\Delta x\rho_{\max}\bigg(\omega_\eta(0) \norm{g}_{\Linf} \norm{v'}_{\Linf}\\
		&&+\norm{g'}_{\Linf}\norm{v'}_{\Linf} ( {2}\omega_\eta(0)\rho_{\max}+\norm{\omega'_\eta}_{\Linf}\norm{\rho}_{\mathbf{L}^1})\bigg) \modulo{\rho^n_{j+1}-\rho^n_j}\\
		&&+3 (\Delta x)^2\left(  \|\omega'\|_{\Linf}\|\rho\|_{\mathbf{L}^1}+ {2}\omega_\eta(0) \rho_{\max} \right)^2 \norm{g}_{\Linf}\norm{v''}_{\Linf}  \modulo{\rho^n_{j-1}}.
	\end{eqnarray*}
	Summing, 
	\begin{eqnarray*}
		\lambda\sum_{j=j_a+1}^{j_b-1}|\mathcal{B}^n_j|
		\leq\Delta t \mathcal{K} \sum_{j=j_a+1}^{j_b-1}\modulo{\rho^n_j-\rho^n_{j-1}}+\Delta t \Delta x \mathcal{K}_1\sum_{j=j_a+1}^{j_b-1}\modulo{\rho^n_{j-1}}+\Delta t \mathcal{K}_2 \sum_{j=j_a+1}^{j_b-1}\modulo{\rho^n_{j+1}-\rho^n_{j}},
	\end{eqnarray*}
	where 
	\begin{eqnarray*}
		\mathcal{K}&=&( {2}\omega_\eta(0)\rho_{\max}+\|\omega'\|_{\Linf}\|\rho\|_{\mathbf{L}^1}) \norm{g}_{\Linf} \norm{v'}_{\Linf},\\
		\mathcal{K}_1&=& \norm{g}_{\Linf}\left(\norm{\omega''}_{\Linf} \norm{\rho}_{\mathbf{L}^1}\norm{v'}_{\Linf}+2  \rho_{\max} \norm{\omega'}_{\mathbf{L}^\infty}\norm{v'}_{\Linf} \right)\\
		&&+3   \left(  \|\omega'\|_{\Linf}\|\rho\|_{\mathbf{L}^1}+ {2}\omega_\eta(0) \rho_{\max} \right)^2 \norm{v''}_{\Linf} \norm{g}_{\Linf},\\
		\mathcal{K}_2&=& \norm{v'}_{\Linf} \rho_{\max}\left(\omega_\eta(0) \norm{g}_{\Linf} +\norm{g'}_{\Linf} ( {2}\omega_\eta(0)\rho_{\max}+\norm{\omega_\eta}_{\Linf}\norm{\rho}_{\mathbf{L}^1})\right).
	\end{eqnarray*}
	We are left with the boundary terms in (\ref{eq:tv_sum}), for $j=j_a,$ we have
	\begin{eqnarray*}
		\rho^{n+1}_{j_a+1}-\rho^{n+1}_{j_a}&=&\rho^n_{j_a+1}-\lambda \left(\rho^n_{j_a+1}g(\rho^n_{j_a+2})v^n_{j_a+3/2}-\rho^n_{j_a}g(\rho^n_{j_a+1})v^n_{j_a+1/2}\right)-\rho^{n+1}_{j_a}\pm \rho^n_{j_a}\\
		&=& (1-\lambda v^n_{j+3/2} g(\rho^n_{j_a+1}))(\rho^n_{j_a+1}-\rho^n_{j_a})-\lambda\rho^n_{j_a}g(\rho^n_{j_a+1})(v^n_{j_a+3/2}-v^n_{j_a+1/2})\\
		&\quad&-\lambda v_{j_a+3/2} \rho^n_{j_a+1}g'(\xi^n_{j_a+3/2})\left(\rho^n_{j_a+2}-\rho^n_{j_a+1}\right)+\rho^n_{j_a}-\rho^{n+1}_{j_a}.
	\end{eqnarray*}
	and  similarly for $j=j_b$
	\begin{eqnarray*}
		\rho^{n+1}_{j_b+1}-\rho^{n+1}_{j_b}&=&\rho^{n+1}_{j_b+1}-\rho^n_{j_b}+\lambda \left(\rho^n_{j_b}g(\rho^n_{j_b+1})v^n_{j_b+1/2}-\rho^n_{j_b-1}g(\rho^n_{j_b})v^n_{j_b-1/2}\right)\pm \rho^n_{j_b+1}\\
		&=&\rho^{n+1}_{j_b+1}-\rho^n_{j_b+1}+(\rho^{n}_{j_b+1}-\rho^n_{j_b})(1+\lambda v^n_{j_b+1/2}\rho^n_{j_b-1}g'(\xi^n_{j_b+1/2}))\\
		&&+\lambda v^n_{j_b+1/2} g(\rho^n_{j_b+1})(\rho^n_{j_b}-\rho^n_{j_b-1})+\lambda \rho^n_{j_b-1}g(\rho^n_{j_b})(v^n_{j_b+1/2}-v^n_{j_b-1/2}).
	\end{eqnarray*}
	Next, Collecting the terms, taking the absolute value and summing over $j$ 
	\begin{eqnarray*}
		&& \sum_{j=j_a}^{j_b}\modulo{\rho^{n+1}_{j+1}-\rho^{n+1}_{j}}\\
		&\leq&\modulo{\rho^{n+1}_{j_a+1}-\rho^{n+1}_{j_a}}+\sum_{j=j_a+1}^{j_b-1}(|\mathcal{A}^n_j|+\lambda \modulo{\mathcal{B}^n_j})+\modulo{\rho^{n+1}_{j_b+1}-\rho^{n+1}_{j_b}}\\
		&\leq& (1-\lambda v^n_{j_a+3/2} g(\rho^n_{j_a+2}))\modulo{\rho^n_{j_a+1}-\rho^n_{j_a}}
		+\lambda\rho^n_{j_a} g(\rho^n_{j_a+1})\modulo{v^n_{j_a+3/2}-v^n_{j_a+1/2}}\\
		&&+\modulo{\rho^n_{j_a}-\rho^{n+1}_{j_a}}-\lambda v^n_{j_a+3/2}\rho^n_{j_a+1}g'(\xi^n_{j_a+3/2})\modulo{\rho^n_{j_a+2}-\rho^n_{j_a+1}}+\sum_{j=j_a+1}^{j_b-1}|\rho_{j+1}^n-\rho_j^n|\\
		&&+\lambda v^n_{j_a+3/2} g(\rho^n_{j_a+2})|\rho_{j_a+1}^n-\rho_{j_a}^n|-\lambda v^n_{j_b+1/2} g(\rho^n_{j_b+1}) |\rho_{j_b}^n-\rho_{j_b-1}^n|\\
		&&+\lambda g'(\xi^n_{j_a+3/2}) \rho^n_{j_a+1}v^n_{j_a+3/2} \modulo{\rho^n_{j_a+2}-\rho^n_{j_a+1}}-\lambda g'(\xi^n_{j_b+1/2}) \rho^n_{j_b-1}v^n_{j_b+1/2} \modulo{\rho^n_{j_b+1}-\rho^n_{j_b}}\\
		&&+\Delta t \mathcal{K} \sum_{j=j_a+1}^{j_b-1}\modulo{\rho^n_j-\rho^n_{j-1}}+\Delta t \Delta x \mathcal{K}_1\sum_{j=j_a+1}^{j_b-1}\modulo{\rho^n_{j-1}}+\Delta t \mathcal{K}_2 \sum_{j=j_a+1}^{j_b-1}\modulo{\rho^n_{j+1}-\rho^n_{j}}\\
		&&+\modulo{\rho^{n+1}_{j_b+1}-\rho^n_{j_b+1}}+\modulo{\rho^{n}_{j_b+1}-\rho^n_{j_b}}(1 +\lambda v^n_{j_b+1/2}\rho^n_{j_b-1}g'(\xi^n_{j_b+1/2}))\\
		&&+\lambda v^n_{j_b+1/2} g(\rho^n_{j_b})\modulo{\rho^n_{j_b}-\rho^n_{j_b-1}}+\lambda \rho^n_{j_b-1}g(\rho^n_{j_b})\modulo{v^n_{j_b+1/2}-v^n_{j_b-1/2}}\\
		&=&\modulo{\rho^n_{j_a}-\rho^{n+1}_{j_a}}+\sum_{j=j_a}^{j_b}|\rho_{j+1}^n-\rho_j^n|+\lambda\rho^n_{j_a} g(\rho^n_{j_a+1})\modulo{v^n_{j_a+3/2}-v^n_{j_a+1/2}} +\modulo{\rho^{n+1}_{j_b+1}-\rho^n_{j_b+1}}\\
		&&+\Delta t \mathcal{K} \sum_{j=j_a+1}^{j_b-1}\modulo{\rho^n_j-\rho^n_{j-1}}+\Delta t \Delta x \mathcal{K}_1\sum_{j=j_a+1}^{j_b-1}\modulo{\rho^n_{j-1}}+\Delta t \mathcal{K}_2 \sum_{j=j_a+1}^{j_b-1}\modulo{\rho^n_{j+1}-\rho^n_{j}}\\
		&&+\lambda \rho^n_{j_b-1}g(\rho^n_{j_b})\modulo{v^n_{j_b+1/2}-v^n_{j_b-1/2}}\\
		&\leq& \modulo{\rho^n_{j_a}-\rho^{n+1}_{j_a}}+(1+2\Delta t \mathcal{K})\sum_{j=j_a}^{j_b}|\rho_{j+1}^n-\rho_j^n|+\modulo{\rho^{n+1}_{j_b+1}-\rho^n_{j_b+1}}+\Delta t \mathcal{K}_3,
	\end{eqnarray*}
	where $\mathcal{K}_3=2 \rho_{\max} \mathcal{K}+\norm{\rho}_{\mathbf{L}^1}\mathcal{K}_1$. By a standard iterative procedure we can deduce, for $1\leq n<N_T-1,$
	\begin{eqnarray*}
		\sum_{j=j_a}^{j_b}\modulo{\rho^{n+1}_{j+1}-\rho^{n+1}_{j}}\leq e^{2\mathcal{K}T}\left( \sum_{j=j_a}^{j_b}|\rho_{j+1}^0-\rho_j^0|+2\frac{C(T)}{q} 
		+\mathcal{K}_3T\right).
	\end{eqnarray*}
	This concludes the proof because $[a,b]\subseteq [x_{j_a}, x_{j_b+1}].$
\end{proof}

\subsection{Discrete Entropy Inequality.}

Next we show that the approximate solution obtained by the scheme \eqref{scheme} fulfills a discrete entropy inequality. Let us define 
$$G_{j+1/2}(u)=u g(u) v_{j+1/2}^n,\qquad \mathcal{F}^{c}_{j+1/2}(u):=G_{j+1/2}(u\vee c)-G_{j+1/2}(u\wedge c)$$
with $a\vee b=\max\{a,b\}$ and $a\wedge b=\min\{a,b\}$.

\begin{lem}
	Let $\rho_j^n$ for $j\in\Z$ and $n\in\N$ given by (\ref{scheme}), and let the CFL condition (\ref{cfl})  and the hypothesis (\ref{hyp_v}) hold. Then  we have 
	\begin{eqnarray}\label{discrentropyinequality}
	&&\modulo{\rho_j^{n+1}-c}-\modulo{\rho_j^{n}-c}\nonumber\\
	&&+\lambda (\mathcal{F}^{c}_{j+1/2}(\rho_j^n)-\mathcal{F}^{c}_{j-1/2}(\rho_{j-1}^n))+\lambda\sign{\rho_j^{n+1}-c} c g(c)(v_{j+1/2}^n-v_{j-1/2}^n)\leq0
	\end{eqnarray}
	for all $j\in\Z$, $n\in\N$ and $c\in[0,\rho_{\max}]$.
\end{lem}
\begin{proof} For a complete proof see \cite[Section 3.4]{2018Gottlich}.
	
\end{proof}

\subsection{Convergence to entropy solution.} \begin{thm}
	Let $\rho_0\in \mathbf{BV}\cap \Linf(\R,[0,\rho_{\max}])$. Let $\Delta x \to 0$ with $\lambda=\frac{\Delta t}{\Delta x}$ constant and satisfying the CFL condition \eqref{cfl}. The sequence of approximate solution $\rho_{\Delta}$ constructed through finite volume scheme \eqref{scheme}-\eqref{numflux} converges in $\mathbf{L}^1_{\mathrm{loc}}$ to a function in $\Linf([0,T]\times\R;[0,\rho_{\max}])$ such that $\norm{\rho}_{\mathbf{L}^1}=\norm{\rho_0}_{\mathbf{L}^1}.$ 
\end{thm} 
\begin{proof}
	Lemma \ref{bounds} ensures that the sequence of approximate solutions $\rho_\Delta$ is bounded in $\Linf.$ Lemma \ref{lem: continuity_in_time} proves the $\mathbf{L}^1-$continuity in time of the sequence $\rho_\Delta,$ while Lemma \ref{lem:spatialbv} guarantees a bound on the spatial total variation in any interval $[a,b]$ not containing $x=0.$ Applying standard compactness results we have that for any interval $[a,b]$ not containing $x=0,$ there exists a subsequence, still denoted by $\rho_\Delta$ converging in $\mathbf{L}^1([0,T]\times[a,b];[0,\rho_{\max}]).$ 
	Let us take a countable set of intervals $[a_i,b_i]$ such that $\cup_i [a_i,b_i]=\R^\ast=\R\setminus\{0\},$ using a standard diagonal process, we can extract a subsequence, still denoted by $\rho_\Delta$, converging in $\mathbf{L}^1_{\mathrm{loc}}([0,T]\times \R;[0,\rho_{\max}])$ and almost everywhere in $[0,T]\times \R,$ to a function $\rho\in\mathbf{L}^\infty([0,T]\times\R;[0,\rho_{\max}]).$
\end{proof}

\begin{lem} 
	\label{LemmaEntropy1}
	Let $\rho(t,x)$ be a weak solution constructed as the limit of approximations $\rho_{\Delta}$ generated by the scheme \eqref{scheme} and let $c\in[0,\rho_{\max}]$. Let $\varphi\in\mathcal{D}(\R^\ast\times[0,T)).$ Then the following entropy inequality is satisfied: 
	\begin{eqnarray}
	&&\int_0^T \int_{\R} (\vert  \rho - c\vert \varphi_t  \mathrm{d}x \mathrm{d}t+  \int_0^T \int_{\R} \sgn(\rho -c) (f(t,x,\rho)-f(t,x,c)) \varphi_x \, \mathrm{d} x \mathrm{d} t\nonumber\\
	&&-  \int_0^T \int_{\R} \sgn(\rho - c) \partial_x f(t,x,c) \varphi \, \mathrm{d}x \mathrm{d}t+\int_{-\infty}^\infty \vert \rho_0(x)- c\vert \varphi(0,x) \mathrm{d}x \geq 0.
	\end{eqnarray}
\end{lem}
\begin{proof}
	Let $\varphi$ be a test function of the type described in the statement of the lemma and set $\varphi^n_j=\varphi(t^n, x_j)$, let us denote $\Delta_- p_j=p_{j}-p_{j-1},$
	we multiply the cell entropy inequality (\ref{discrentropyinequality}) by $\varphi^n_j \Delta x,$ and then sum by parts to get 
	\begin{eqnarray*}
		\label{eq:S1}
		S_1+S_2+S_3&=& \Delta x \Delta t \sum_{n\geq 0} \sum_{j\in \Z}\modulo{\rho^{n+1}_j-c} (\varphi^{n+1}_j-\varphi^n_j)/\Delta t+\Delta x \sum_j \modulo{\rho^0_j-c} \varphi^0_j \\
		\label{eq:S2}
		&& +\Delta x \Delta t \sum_{n\geq 0} \sum_{j\in \Z}  \mathcal{F}^c_{j-1/2}\Delta_-\varphi^n_j/\Delta x\\
		\label{eq:S3}
		&& -\Delta x \Delta t \sum_{n\geq 0} \sum_{j\in \Z}\sign{\rho^{n+1}_j-c} c\,g(c)\, \Delta_- v_{j+1/2} \,\varphi^n_j / \Delta x \geq 0.
	\end{eqnarray*}
	By Lebesgue's dominated convergence theorem as $\Delta:=(\Delta x, \Delta t) \to 0$, 
	\begin{eqnarray*}
		S_1   \rightarrow \int_0^T \int_{\R} \modulo{\rho-c} \varphi_t \mathrm{d}x \mathrm{d}t+\int_{-\infty}^\infty \vert \rho_0(x)-c\vert \varphi(0,x) \mathrm{d}x,
	\end{eqnarray*}
	and 
	\begin{eqnarray*}
		S_2  \rightarrow \int_0^T \int_{\R} \sign{\rho -c} (f(t,x,\rho)-f(t,x,c)) \varphi_x \, \mathrm{d}x \mathrm{d}t.
	\end{eqnarray*}
	Now let us study the sum $S_3$ and we have 
	\begin{eqnarray*}\label{eq:s33}
		S_3&=&-\Delta x \Delta t \sum_{n\geq 0}\sum_{{j\in \Z\,\, j\leq-1}} \sign{\rho^{n+1}_{j}-c} c g(c)\, \Delta_{-} v_{j+1/2}  \,\varphi^n_j/\Delta x\\
		&&-\Delta x \Delta t \sum_{n\geq 0}\sum_{{j\in \Z\,\, j\geq 1}} \sign{\rho^{n+1}_{j}-c} c g(c)\, \Delta_- v_{j+1/2}  \,\varphi^n_j/\Delta x
		\label{eq:s34}\\
		&&-\Delta x \Delta t \sum_{n\geq 0} \sign{\rho^{n+1}_{0}-c} c g(c)\, \Delta_- v_{1/2}  \,\varphi^n_0/\Delta x\label{eq:s35},\\
		&=&S_{31}+S_{32}+S_{33}.
	\end{eqnarray*}
	Observe that the support of the test function $\varphi$ does not include the discontinuity flux point $0,$ for this reason we consider $\varphi_0=0$ according to our discretization, then the sum  $S_{33}$ is equal to zero because $\varphi_0=0.$  
	Finally, 
	\begin{eqnarray*}
		S_{31}+S_{32} \rightarrow  & -\int_0^T \int_{-\infty}^0 \sign{\rho - c} \partial_x f(t,x,c) \varphi \, \mathrm{d}x \mathrm{d}t\\
		&-\int_0^T \int_{0}^\infty \sign{\rho - c} \partial_x f(t,x,c) \varphi \, \mathrm{d}x \mathrm{d}t.
	\end{eqnarray*}
\end{proof}
\begin{lem}\label{LemmaEntropy2}
	Let $\rho(t,x)$ be a weak solution constructed as the limit of approximations $\rho_{\Delta}$ generated by the scheme \eqref{scheme} and let $c\in[0,\rho_{\max}]$. Let $\varphi\in\Cc1(\R\times[0,T)).$ Then the following entropy inequality is satisfied: 
	\begin{eqnarray*}
		&&	\int_0^{T}\int_{\R} \abs{\rho-c}\varphi_t+\sign{\rho-c}(f(t,x,\rho)-f(t,x,c))\, \partial_x \varphi\, \mathrm{d}x\,\mathrm{d}t\\
		&&+\int_0^{T}\int_{\R_\ast} |\partial_x f(t,x,c)| \varphi\, \mathrm{d}x\,\mathrm{d}t+ \int_{\R} \abs{\rho_0(x)-c}\varphi(0,x) \mathrm{d}x \\
		&&+ \int_{0}^{T}  \abs{(k_r-k_l) c \,g(c)\, \psi(\rho\ast\omega_\eta )} \varphi(t,0) \mathrm{d}t\geq0.
	\end{eqnarray*}
\end{lem}
\begin{proof}
	Let $\varphi$ be a test function of the type described in the statement of the lemma and set $\varphi^n_j=\varphi(t^n, x_j)$.    There exist $T>0$ and $R>0$ such that $\varphi(t,x)=0$ for $t>T$ and $\vert x \vert >R$. 
	Our starting point is the following cell entropy inequality which is a consequence of (\ref{discrentropyinequality}).
	\begin{equation}\label{eq:discrentropy2}
	\modulo{\rho^{n+1}_j-c}\leq \modulo{\rho^{n}_j-c}-\lambda \Delta_- \mathcal{F}^c_{j+1/2}+\lambda \modulo{c g(c) \Delta_-  v^n_{j+1/2}}
	\end{equation}
	We multiply (\ref{eq:discrentropy2}) by $\varphi^n_j \Delta x,$ and then sum by parts to get 
	\begin{eqnarray*}
		\label{eq:2S1}
		S_4+S_5+S_6&=&\Delta x \Delta t \sum_{n\geq 0} \sum_{j\in \Z}\modulo{\rho^{n+1}_j-c}(\varphi^{n+1}_j-\varphi^n_j)/\Delta t+\Delta x \sum_j \modulo{\rho^0_j-c} \varphi^0_j \\
		\label{eq:2S2}
		&&  +\Delta x \Delta t \sum_{n\geq 0} \sum_{j\in \Z}  \mathcal{F}^c_{j-1/2}(\Delta_-\varphi^n_j/\Delta x)\\
		\label{eq:2S3}
		&& +\Delta x \Delta t \sum_{n\geq 0} \sum_{j\in \Z}  \modulo{ c g(c)\Delta_- v^n_{j+1/2}}\varphi^n_j / \Delta x \geq 0.
	\end{eqnarray*}
	By Lebesgue's dominated convergence theorem as $\Delta:=(\Delta x, \Delta t) \to 0$, 
	\begin{eqnarray*}
		S_4   \rightarrow \int_0^T\int_{\R} \modulo{\rho-c} \varphi_t \mathrm{d}x \mathrm{d}t  +\int_{-\infty}^\infty \vert \rho_0(x)-c\vert \varphi(0,x) \mathrm{d}x.
	\end{eqnarray*}
	Following the same standard arguments as in Lemma (\ref{LemmaEntropy1}), 
	\begin{eqnarray*}
		S_5   \rightarrow		\int_0^{T}\int_{\R} \sign{\rho-c}(f(t,x,\rho)-f(t,x,c))\, \partial_x \varphi\, \mathrm{d}x\,\mathrm{d}t.    
	\end{eqnarray*}
	Now we can rewrite the sum $S_6$
	\begin{eqnarray*}
		\label{eq:1}
		S_6&=&\Delta x \Delta t  \sum_{n\geq 0}\sum_{{j\in \Z\,\j\leq -1}}  \modulo{ c g(c)\Delta_- v^n_{j+1/2}}\varphi^n_j / \Delta x\\
		\label{eq:2}
		&&+\Delta x \Delta t  \sum_{n\geq 0}\sum_{{j\in \Z\,\j\geq 1}}  \modulo{ c g(c)\Delta_- v^n_{j+1/2}}\varphi^n_j / \Delta x\\
		\label{eq:3}
		&&+\Delta t  \sum_{n\geq 0} \modulo{ c g(c)\Delta_- v^n_{1/2}}\varphi^n_0\\
		&=&S_{61}+S_{62}+S_{63} .
	\end{eqnarray*}
	At this point, we can observe that as $\Delta:=(\Delta x, \Delta t)\to 0$
	\begin{eqnarray*}
		S_{61}+S_{62} \rightarrow 	\int_0^{T}\int_{\R\setminus \{0\}} |f(t,x,c)_x| \varphi\, \mathrm{d}x\,\mathrm{d}t\\
		S_{63}\rightarrow \int_{0}^{T}  \abs{(k_r-k_l) c\, g(c)\, \psi(\rho\ast\omega_\eta )} \varphi(t,0) \mathrm{d}t.
	\end{eqnarray*}
\end{proof}
\begin{thm}
	Let $\rho(t,x)$ be the limit of approximations $\rho_{\Delta}$ generated by the scheme \eqref{scheme} and let $c\in[0,\rho_{\max}]$. Then $\rho(t,x)$ is an entropy solution satisfying the Definition \ref{def:solution}.
\end{thm}
\begin{proof}
	Let $0\leq \varphi \in \Cc1([0,T)\times \R).$  We set $\varphi^n_j=\varphi(t^n,x_j).$
	For $\varepsilon>0,$ define the set 
	\begin{eqnarray*}
		\sigma^\varepsilon_0=\{(t,x)\in [0,T)\times \R| x\in(-\varepsilon,\varepsilon),\,t\in[0,T)\}.
	\end{eqnarray*}
	For each sufficiently small $\varepsilon>0$ we can write the test function $\varphi$ as a sum of two test functions, one having support away from 0 and the other with support in $\sigma_0^\varepsilon.$ We take test functions $\psi^{\varepsilon},\,\alpha^{\varepsilon}\in \Cc1([0,T)\times \R)$ such that 
	\begin{eqnarray*}
		\varphi(t,x)=\psi^{\varepsilon}(t,x)+\alpha^{\varepsilon}(t,x), \quad 0\leq\psi^{\varepsilon}(t,x)\leq \varphi(t,x), \quad  0\leq\alpha^{\varepsilon}(t,x)\leq \varphi(t,x),
	\end{eqnarray*}
	where $\psi^\varepsilon$ has support located around the jump in 0 
	\begin{eqnarray*}
		\supp(\psi^\varepsilon)\subseteq \sigma^\varepsilon_0, \qquad
		\psi^\varepsilon(t,0)=\varphi(t,0),
	\end{eqnarray*}
	and $\alpha^\varepsilon$ vanishes around the jump, i.e.
	\begin{eqnarray*}
		\supp(\alpha^\varepsilon)\subseteq [0,T)\times \R^\ast. 
	\end{eqnarray*}
	We can take this decomposition in such way that 
	\begin{equation}\label{eq:limit}
	\alpha^\varepsilon\rightarrow\varphi \quad \hbox{ in } \mathbf{L}^1([0,T)\times \R ), \quad \psi^\varepsilon\rightarrow 0 \quad \hbox{ in } \mathbf{L}^1([0,T)\times \R )
	\end{equation}
	as $\varepsilon\to 0.$
	By applying Lemma \ref{LemmaEntropy1} with the test function $\alpha^\varepsilon$ and Lemma \ref{LemmaEntropy2} with $\psi^\varepsilon$, and summing the two entropy inequalities, using $\varphi=\psi^\varepsilon+\alpha^\varepsilon$ along with $\psi^\varepsilon(0,t)=\varphi(0,t)$ to get 
	\begin{eqnarray*}
		&&\int_{0}^T \int_\R (\vert  \rho - c\vert \varphi_t  \mathrm{d}x \mathrm{d}t + \int_{0}^T \int_\R \sign{\rho -c} (f(t,x,\rho)-f(t,x,c))\varphi_x  \mathrm{d}x \mathrm{d}t\\
		&&-\int_{0}^T \int_\R \sign{\rho -c}  f(t,x,c)_x \alpha^\varepsilon \mathrm{d}x \mathrm{d}t + \int_{0}^T \int_{\R^\ast} \modulo{  f(t,x,c))_x } \psi^\varepsilon \mathrm{d}x \mathrm{d}t\\ 
		&&+\int_{0}^{T}  \abs{(k_r-k_l) c\,g(c)\, \psi(\rho\ast\omega_\eta )} \varphi(t,0) \mathrm{d}t+\int_{-\infty}^\infty \vert \rho_0(x)-c\vert \phi(0,x) \mathrm{d}x\geq 0.
	\end{eqnarray*}
	Thanks to (\ref{eq:limit}), we can complete the proof by sending $\varepsilon\to 0.$
\end{proof}

\subsection{$\mathbf{L}^1$-Stability and uniqueness.}
\begin{thm}
	Assume the hypothesis \eqref{hyp_v}. If $\rho$ and $\tilde\rho$ are two entropy solutions of  \eqref{eq:CL} in the sense of Definition \ref{def:solution}, the following inequality holds
	\begin{equation}
	\label{eq:stability}
	\norm{\rho(t,\cdot)-\tilde \rho(t,\cdot)}_{\mathbf{L}^1(\R)}\le  e^{K(T) t}\norm{\rho(0,\cdot)-\tilde \rho(0,\cdot)}_{\mathbf{L}^1(\R)},
	\end{equation}
	for almost every $0<t<T$ and some suitable constant $K(T)>0$.
\end{thm}

\begin{proof}
	Following \cite[Theorem 2.1]{KarlsenRisebroTowers}, for any two entropy solutions $\rho$ and $\tilde \rho$ we can derive the $L^1$ contraction property through the doubling of variables technique:
	\begin{eqnarray}\label{eq:inequality}
	&&\iint_{\R^{+}\times\R} \left(\modulo{\rho-\tilde \rho}\phi_t+
	\sign{ \rho-\tilde \rho}(f(t,x,\rho)-f(t,x,\tilde{\rho}))\phi_x\right) \mathrm{d} x \mathrm{d} t\nonumber \\ 
	&&\leq K \iint_{\R^+\times\R} \modulo{\rho-\tilde \rho} \phi \mathrm{d} x \mathrm{d} t,
	\end{eqnarray}
	where $K=K(T),$ for any $0\leq \phi \in \mathcal D{(\R^+\times \R^*)}.$
	We remove the assumption in (\ref{eq:inequality}) that $\phi$ vanishes near $0,$ by introducing the following Lipschitz function for $h>0$
	\begin{eqnarray*}
		\mu_h(x)=\begin{cases}
			\frac{1}{h}(x+2h), \quad x\in[-2h,-h],\\
			1,\quad x\in[-h,h],\\
			\frac{1}{h}(2h-x), \quad x\in[h,2h],\\
			0, \quad|x|\geq 2h.
		\end{cases}
	\end{eqnarray*}
	Now we can define $\Psi_h(x)=1-\mu_h(x),$ noticing that $\Psi_h-1\to 0$ in $L^1$ as $h\to 0.$ Moreover, $\Psi_h$ vanishes in a neighborhood of $0.$ For any $0\leq \Phi \in \Cc\infty(\R^+\times\R),$ we can check that $\phi=\Phi\Psi_h$ is an admissible test function for \eqref{eq:inequality}. Using $\phi$ in \eqref{eq:inequality} and integrating by parts we get 
	\begin{eqnarray*}
		&&\iint_{\R^+\times\R} \left(\modulo{\rho- \tilde \rho}\Phi_t\Psi_h+\sign{\rho- \tilde \rho}(f(t,x,\rho)-f(t,x,\tilde \rho))\Phi_x\Psi_h \right) \mathrm{d} x \mathrm{d} t\\ 
		&&-\underbrace{\iint_{\R^+\times\R}\sign{\rho-\tilde \rho)(f(t,x,\rho}-f(t,x,\tilde \rho))\Phi(t,x)\Psi'_h(x) \mathrm{d} x \mathrm{d} t}_{J(h)}\\
		&&\leq K \iint_{\R^+\times\R} \modulo{\rho-\tilde \rho} \Phi \Psi_h \mathrm{d} x \mathrm{d} t. 
	\end{eqnarray*}
	Sending $h\to 0$ we end up with 
	\begin{eqnarray*}
		&\iint_{\R^+\times\R} \left(\modulo{\rho-\tilde \rho}\Phi_t+\sign{ \rho-\tilde \rho}(f(t,x,\rho)-f(t,x,\tilde \rho))\Phi_x\right) \mathrm{d} x \mathrm{d} t\\ 
		&\qquad \leq K \iint_{\R^+\times\R} \modulo{\rho-\tilde \rho} \Phi \mathrm{d} x \mathrm{d} t+\lim_{h\to 0}J(h). 
	\end{eqnarray*}
	We can write 
	\begin{eqnarray*}
		\lim_{h\to 0} J(h) &=& \lim_{h \to 0} \frac{1}{h} \int_0^T\int_h^{2h} \sign{\rho-\tilde \rho}(f(t,x,\rho)-f(t,x,\tilde \rho)) \mathrm{d} x \mathrm{d} t\\
		&&- \lim_{h \to 0} \frac{1}{h} \int_0^T\int_{-2h}^{-h} \sign{\rho-\tilde \rho}(f(t,x,\rho)-f(t,x,\tilde \rho)) \mathrm{d} x \mathrm{d} t\\
		&=&\int_0^T  [\sign{\rho-\tilde \rho}(f(t,x,\rho)-f(t,x,\tilde \rho))]_{x=0^-}^{x=0^+} \Phi(t,0) \mathrm{d} t,
	\end{eqnarray*}
	where we indicate the limits from the right and left at $x=0.$
	The aim is to prove that the limit $\displaystyle{\lim_{h \to 0}}J(h)\leq 0.$ This is equivalent to prove that the quantity
	\begin{eqnarray*}
		S:=[\sign{\rho-\tilde \rho}(f(t,x,\rho)-f(t,x,\tilde \rho))]_{x=0^-}^{x=0^+}\leq 0.
	\end{eqnarray*}
	
	A simple application of the Rankine-Hugoniot condition yields $S\leq 0,$ see the proof of \cite[Theorem 2.1]{KarlsenRisebroTowers}, noticing that in this setting there is no flux crossing. 
	Therefore we conclude that $S\leq 0.$
	In this way we know that (\ref{eq:inequality}) holds for any $0\leq\phi\in \Cc\infty (\R^+\times\R).$ For $r>1,$ let $\gamma_r:\R \to \R$ be a $C^\infty$ function which takes values in $[0,1]$ and satisfies 
	\begin{eqnarray*}
		\gamma_r(x)=\begin{cases}
			1, \quad \modulo{x}\leq r,\\
			0, \quad \modulo{x}\geq r+1.
		\end{cases}
	\end{eqnarray*}
	Fix $s_0$ and $s$ such that $0<s_0<s<T.$ For any $\tau>0$ and $k>0$ with $0<s_0+\tau<s+k<T,$ let $\beta_{\tau,k}:[0,T]\to\R$ be a Lipschitz function that is linear on $[s_0,s_0+\tau[\cup[s,s+k]$ and satisfies 
	\begin{eqnarray*}
		\beta_{\tau,k}(t)=
		\begin{cases}
			0, \quad t\in[0,s_0]\cup[s+k,T],\\
			1,  \quad t\in[s_0+\tau,s].
		\end{cases}
	\end{eqnarray*}
	We can take the admissible test function via a standard regularization argument\\
	$\phi=\gamma_r(x)\beta_{\tau,k}(t).$ Using this test function in (\ref{eq:inequality}) we obtain 
	\begin{eqnarray*}
		&&\frac{1}{k} \int_s^{s+k} \int_{\R} \modulo{\rho(t,x)-\tilde \rho(t,x)} \gamma_r(x) \mathrm{d} x \mathrm{d} t  -\frac{1}{\tau} \int_{s_0}^{s_0+k} \int_\R  \modulo{\rho(t,x)-\tilde \rho(t,x)} \gamma_r(x) \mathrm{d} x \mathrm{d} t\\
		&&\leq K \int_{s_0}^{s_0+k} \int_\R \modulo{\rho-\tilde\rho} \gamma_r(x) \mathrm{d} x \mathrm{d} t\\
		&&\quad+\norma{\gamma'_r}_\infty \int_{s_0}^{s+k} \int_{r\leq \modulo{x}\leq r+1}\sign{\rho-\tilde \rho}(f(t,x,\rho)-f(t,x,\tilde \rho)) \mathrm{d} x \mathrm{d} t. 
	\end{eqnarray*}
	Sending $s_0\to 0,$ we get 
	\begin{eqnarray*}
		\frac{1}{k} \int_s^{s+k} \int_{-r}^{r}\modulo{\rho(t,x)-\tilde \rho(t,x)} \gamma_r(x) \mathrm{d} x \,\mathrm{d} t&\leq& \int_{-r}^{r} \modulo{ \rho_0(x)- \tilde \rho_0(x)} \mathrm{d} x\\
		&&+ \frac{1}{\tau} \int_{0}^{\tau} \int_{-r}^{r} \modulo{\tilde \rho(t,x)-\tilde \rho_0(x)} \mathrm{d} x \mathrm{d} t\\
		&&+\frac{1}{\tau} \int_{0}^{\tau} \int_{-r}^{r} \modulo{ \rho(t,x)-\rho_0(x)} \mathrm{d} x \,\mathrm{d} t\\
		&&+ K \int_{0}^{t+\tau} \int_\R \modulo{\rho-\tilde \rho} \gamma_r(x) \mathrm{d} x \,\mathrm{d} t + o\left(\frac{1}{r}\right).
	\end{eqnarray*}
	Observe that the second and the third terms on the right-hand side of the inequality tends to zero as $\tau\to 0$ following the same argument in \cite[Lemma B.1]{KarlsenRisebroTowers}, because our initial condition is satisfied in the ``weak" sense of the definition of our entropy condition. Sending $\tau\to 0$ and $r\to\infty,$ we have 
	\begin{eqnarray*}
		\frac{1}{k} \int_s^{s+k} \int_\R \modulo{\rho(t,x)-\tilde \rho(t,x)} \mathrm{d} x \,\mathrm{d} t &\leq& \int_\R \modulo{\rho_0(x)-\tilde \rho_0(x)}\mathrm{d} x\\
		&&+K \int_0^{s+k} \int_\R  \modulo{\rho(t,x)-\tilde \rho(t,x)} \mathrm{d} x \,\mathrm{d} t.
	\end{eqnarray*}
	Sending $k\to0$ and an application of Gronwall's inequality give us the statement.
\end{proof}

\begin{lem}[{\bf A Kru\v{z}kov-type integral inequality}]
	For any two entropy solutions $\rho=\rho(t,x)$ and $\tilde \rho=\tilde \rho(t,x)$
	the integral inequality \eqref{eq:inequality} holds for any $0\leq \phi\in \Cc\infty(\R^+\times \R\setminus\{0\}).$  
\end{lem}
\begin{proof}
	Let $0\leq\phi\in \Cc\infty\left((\R^+\times \R\setminus\{0\})^2\right),\: \phi=\phi(t,x,s,y), \: \rho=\rho(t,x)$ and $\tilde \rho=\tilde \rho(s,y).$
	From the definition of entropy solution for $\rho=\rho(t,x)$ with $\kappa=\tilde \rho(s,y)$ we get 
	\begin{eqnarray*}
		&&-\iint_{\R^+\times \R} \left(\modulo{\rho-\tilde \rho}\phi_t+\sign{\rho-\tilde \rho}\left(f(t,x,\rho)-f(t,x,\tilde \rho)\right)\phi_x\right)\, \mathrm{d} t\, \mathrm{d} x \\
		&&+\iint_{\R^+\times \R\setminus\{0\}} \sign{\rho-\tilde \rho} f(t,x,\tilde \rho)_x \phi \,\mathrm{d} t\, \mathrm{d} x \leq 0.
	\end{eqnarray*}
	Integrating over $(s,y)\in \R^+\times \R,$ we find
	\begin{eqnarray}\label{eq:1ineq} 
	&&-{\iint\iint}_{(\R^+\times \R)^2} \left(\modulo{\rho-\tilde \rho}\phi_t+\sign{\rho-\tilde \rho}\left(f(t,x,\rho)-f(t,x,\tilde \rho)\right)\phi_x\right)\, \mathrm{d} t \,\mathrm{d} x \,\mathrm{d} s\, \mathrm{d} y\nonumber\\ 
	&&+{\iint\iint}_{(\R^+\times \R\setminus\{0\})^2} \sign{\rho-\tilde \rho} f(t,x,\tilde \rho)_x \phi\, \mathrm{d} t\, \mathrm{d} x\, \mathrm{d} s\, \mathrm{d} y\leq 0.
	\end{eqnarray}
	Similarly, for the entropy solution $\tilde \rho=\tilde \rho(s,y)$ with $\alpha(y)=\rho(t,x)$
	\begin{eqnarray}\label{eq:2ineq} 
	&&-{\iint\iint}_{(\R^+\times \R)^2} \left(\modulo{\tilde \rho-\rho}\phi_s+\sign{\tilde \rho-\rho}\left(f(s,y,\tilde \rho)-f(s,y,\rho)\right)\phi_x\right)\, \mathrm{d} t\, \mathrm{d} x\, \mathrm{d} s\, \mathrm{d} y\nonumber\\ 
	&&+{\iint\iint}_{(\R^+\times \R-\{0\})^2} \sign{\rho-\tilde \rho} f(t,x,\tilde \rho)_x \phi\,\mathrm{d} t\, \mathrm{d} x\, \mathrm{d} s\, \mathrm{d} y \leq 0.
	\end{eqnarray}
	
	Note that we can write, for each $(t,x)\in \R^+\times \R\setminus\{0\},$
	\begin{eqnarray*}
		&&\sign{\rho-\tilde \rho}(f(t,x,\rho)-f(t,x,\tilde \rho))  \phi_x -\sign{\rho-\tilde \rho} f(t,x,\tilde \rho)_x \phi\\
		&&=\sign{\rho-\tilde \rho}(f(t,x,\rho)-f(s,y,\tilde \rho))\phi_x-\sign{\rho-\tilde \rho}\left[(f(t,x,\tilde \rho)-f(s,y,\tilde \rho))\phi\right]_x,
	\end{eqnarray*}
	so that
	\begin{eqnarray*}
		&&\quad-{\iint\iint}_{(\R^+\times \R)^2}\sign{\rho-\tilde \rho}(f(t,x,\rho)-f(t,x,\tilde \rho))  \phi_x \,\mathrm{d} t\,\mathrm{d} x\,\mathrm{d} s\,\mathrm{d} y \\
		&&\quad+{\iint\iint}_{(\R^+\times \R\setminus\{0\})^2}\sign{\rho-\tilde \rho} f(t,x,\tilde \rho)_x \phi \,\mathrm{d} t\,\mathrm{d} x\,\mathrm{d} s\,\mathrm{d} y\\
		&&=-{\iint\iint}_{(\R^+\times \R)^2}\sign{\rho-\tilde \rho}(f(t,x,\rho)-f(s,y,\tilde \rho))\phi_x\,\mathrm{d} t\,\mathrm{d} x\,\mathrm{d} s\,\mathrm{d} y\\
		&&\quad+ {\iint\iint}_{(\R^+\times \R\setminus\{0\})^2} \sign{\rho-\tilde \rho}\left[(f(t,x,\tilde \rho)-f(s,y,\tilde \rho))\phi\right]_x \,\mathrm{d} t\,\mathrm{d} x\,\mathrm{d} s\,\mathrm{d} y.
	\end{eqnarray*}
	Similarly, writing, for each $(y,s)\in \R^+\times \R\setminus\{0\}$
	\begin{eqnarray*}
		&&\quad \sign{\tilde \rho-\rho}(f(s,y,\tilde \rho)-f(s,y,\rho))  \phi_y -\sign{\tilde \rho- \rho} f(s,y,\rho)_y \phi\\
		&&=\sign{\rho-\tilde \rho}(f(s,y,\tilde \rho)-f(s,y,\rho))\phi_y-\sign{\rho-\tilde \rho}\left[(f(t,x, \rho)-f(s,y, \rho))\phi\right]_x,
	\end{eqnarray*}
	so that
	\begin{eqnarray*}
		&&\quad -{\iint\iint}_{(\R^+\times \R)^2}\sign{\rho-\tilde \rho}(f(s,y,\tilde \rho)-f(s,y, \rho))  \phi_y \,\mathrm{d} t\,\mathrm{d} x\,\mathrm{d} s\,\mathrm{d} y \\
		&&\quad+{\iint\iint}_{(\R^+\times \R\setminus\{0\})^2}\sign{\rho-\tilde \rho} f(s,y, \rho)_y \phi \,\mathrm{d} t\,\mathrm{d} x\,\mathrm{d} s\,\mathrm{d} y\\
		&&=-{\iint\iint}_{(\R^+\times \R)^2}\sign{\rho-\tilde \rho}(f(t,x,\tilde \rho)-f(s,y,\rho))\phi_x\,\mathrm{d} t\,\mathrm{d} x\,\mathrm{d} s\,\mathrm{d} y\\
		&&\quad+ {\iint\iint}_{(\R^+\times \R\setminus\{0\})^2} \sign{\rho-\tilde \rho}\left[(f(t,x, \rho)-f(s,y, \rho))\phi\right]_y \,\mathrm{d} t\,\mathrm{d} x\,\mathrm{d} s\,\mathrm{d} y.
	\end{eqnarray*}
	Let us introduce the notations 
	\begin{eqnarray*}
		\partial_{t+s}&=&\partial_t+\partial_s, \quad \partial_{x+y}=\partial_x+\partial_y,\\
		\partial^2_{x+y}&=&(\partial_x+\partial_y)^2=\partial_x^2+2\partial_x\partial_y+\partial_y^2.
	\end{eqnarray*}
	Adding \eqref{eq:1ineq} and \eqref{eq:2ineq} we obtain 
	\begin{eqnarray}\label{eq:A22}
	&& -{\iint\iint}_{(\R^+\times \R)^2} \left(\modulo{\rho-\tilde \rho}\partial_{t+s}\phi+\sign{\rho-\tilde \rho}\left(f(t,x,\rho)-f(s,y,\tilde \rho)\right)\partial_{x+y} \phi\right)   \,\mathrm{d} t\,\mathrm{d} x\,\mathrm{d} s\,\mathrm{d} y\nonumber\\ 
	&&+{\iint\iint}_{\R^+\times \R\setminus\{0\}} \sign{\rho-\tilde \rho} \left(\partial_x\left[ (f(t,x,\tilde \rho)-f(s,y,\tilde \rho)) \phi \right] 
	\right.\nonumber\\  
	&&\quad\left.+ \partial_y \left[(f(t,x,\rho)-f(s,y,\rho)) \phi \right]\right)  \,\mathrm{d} t\,\mathrm{d} x\,\mathrm{d} s\,\mathrm{d} y \leq 0.
	\end{eqnarray}
	We introduce a non-negative function $\delta\in \Cc\infty(\R),$ satisfying $\delta(\sigma)=\delta(-\sigma),\, \delta(\sigma)=0$ for $\modulo{\sigma}\geq 1,$ and $\int_{\R} \delta(\sigma) d \sigma=1.$ For $u>0$ and $z\in\R,$ let $\delta_{p}(z)=\frac{1}{p}\delta(\frac{z}{p}).$ We take our test function $\phi=\phi(t,x,s,y)$ to be of the form 
	\begin{eqnarray*}
		\Phi(t,x,s,y)= \phi\left(\frac{t+s}{2},\frac{x+y}{2}\right)\delta_p\left(\frac{x-y}{2}\right) \delta_p\left(\frac{t-s}{2} \right),
	\end{eqnarray*}
	where $0\leq \phi\in \Cc\infty\left(\R^+\times \R\setminus{0}\right)$ satisfies  
	\begin{eqnarray*}
		\phi(t,x)=0, \quad \forall(t,x)\in[-h,h]\times[0,T],
	\end{eqnarray*}
	for small $h>0.$ By making sure that $p<h, $
	one can check that $\Phi$ belongs to $\Cc\infty\left(\left(\R^+\times \R\setminus\{0\}\right)^2\right).$
	We have 
	\begin{eqnarray*}
		\partial_{t+s} \Phi(t,x,s,y)= 	\partial_{t+s} \phi\left( \frac{t+s}{2},\frac{x+y}{2}\right)\delta_p\left(\frac{x-y}{2}\right) \delta_p\left(\frac{t-s}{2} \right),\\
		\partial_{x+y} \Phi(t,x,s,y)= 	\partial_{x+y} \phi\left(\frac{t+s}{2},\frac{x+y}{2}\right)\delta_p\left(\frac{x-y}{2}\right) \delta_p\left(\frac{t-s}{2} \right),
	\end{eqnarray*}
	and using $\Phi$ as test function in (\ref{eq:A22}) 
	\begin{eqnarray*}
		&& -{\iint\iint}_{(\R^+\times \R)^2} \left(I_1(t,x,s,y)+I_2(t,x,s,y)\right)\delta_p\left(\frac{x-y}{2}\right) \delta_p\left(\frac{t-s}{2} \right) \,\mathrm{d} t\,\mathrm{d} x\,\mathrm{d} s\,\mathrm{d} y\\
		&&\leq {\iint\iint}_{(\R^+\times \R\setminus\{0\})^2} \left(I_3(t,x,s,y)+I_4(t,x,s,y)+I_5(t,x,s,y)\right) \,\mathrm{d} t\,\mathrm{d} x\,\mathrm{d} s\,\mathrm{d} y,
	\end{eqnarray*}
	where 
	\begin{eqnarray*}
		I_1&=&\modulo{\rho(t,x)-\tilde \rho(s,y)} \partial_{t+s} \phi\left(\frac{t+s}{2},\frac{x+y}{2}\right),\\
		I_2&=&\sign{\rho(t,x)-\tilde \rho(s,y)} (f(t,x,\rho)-f(s,y,\tilde \rho)) \partial_{x+y} \phi\left(\frac{t+s}{2}, \frac{x+y}{2}\right),\\
		I_3&=&-\sign{\rho(t,x)-\tilde \rho(s,y)} \left(\partial_x f(t,x,\tilde \rho) -\partial_y f(s,y,\rho)\right)\phi\left(\frac{t+s}{2},\frac{x+y}{2},\right) \\
		&&\hspace{4em}	\hspace{4em}\times\delta_p\left(\frac{x-y}{2}\right)\delta_p\left(\frac{t-s}{2} \right),\\
		I_4&=&-\sign{\rho(t,x)-\tilde \rho(s,y)} \delta_p\left(\frac{x-y}{2}\right) \delta_p\left(\frac{t-s}{2} \right)\\
		&&	\hspace{4em}	\hspace{4em}\times\left[ \partial_x\phi\left(\frac{t+s}{2},\frac{x+y}{2},\right) (f(t,x,\tilde \rho) -f(s,y, \tilde \rho))\right.\\
		&&\hspace{4em}	\hspace{4em}		\left.		\times \partial_y\phi\left(\frac{t+s}{2},\frac{x+y}{2}\right) (f(t,x, \rho) -f(s,y,\rho))\right],\\
		I_5&=&\left(F(x,\rho(t,x),\tilde \rho(s,y))-F(y,\rho(t,x),\tilde \rho(s,y))\right)\phi\left(\frac{t+s}{2},\frac{x+y}{2}\right) \\
		&&\hspace{4em}\hspace{4em}\times\partial_x\delta_p\left(\frac{x-y}{2}\right) \delta_p\left(\frac{t-s}{2} \right),
	\end{eqnarray*}
	where $F(x,\rho,c):=\sign{(\rho-c)}\left(f(t,x,\rho)-f(t,x,c)\right).$ We now use the change of variables 
	\begin{eqnarray*}
		\tilde x =\frac{x+y}{2}, \quad 	\tilde t =\frac{t+s}{2}, \quad z =\frac{x-y}{2},\quad \tau =\frac{t-s}{2},
	\end{eqnarray*}
	which maps $(\R^+\times \R)^2$ in $\Omega\subset \R^4$  and $(\R^+\times \R\setminus\{0\})^2$ in $\Omega_0\subset \R^4,$ where
	\begin{eqnarray*}
		&	\Omega=\{(\tilde x, \tilde t,z, \tau)\in\R^4 :\,0<\tilde t \pm \tau<T\}, \qquad \Omega_0=\{(\tilde x, \tilde t,z, \tau)\in\Omega : \tilde x \pm z\neq 0\},
	\end{eqnarray*}
	respectively. 
	With this changes of variables  {we can rewrite}
	\begin{eqnarray*}
		&\partial_{t+s} \phi\left(\frac{t+s}{2},\frac{x+y}{2}\right)=\partial_{\tilde t} \phi(\tilde t, \tilde x),\qquad 
		\partial_{x+y} \phi\left(\frac{t+s}{2},\frac{x+y}{2}\right)=\partial_{\tilde x} \phi(\tilde t, \tilde x).	\end{eqnarray*}
	Now we can write 
	\begin{eqnarray*}
		&&-{\iint\iint}_{\Omega} \left(I_1(\tilde t, \tilde x,\tau,z)+I_2(\tilde t, \tilde x,\tau,z)\right)\delta_p\left(z\right) \delta_p\left(\tau \right) \,\mathrm{d}\tilde t\,\mathrm{d} \tilde x\,\mathrm{d}\tau\,\mathrm{d} z\\
		&&\quad\leq {\iint\iint}_{\Omega_0} \left(I_3(\tilde t, \tilde x,\tau,z)+I_4(\tilde t, \tilde x,\tau,z)+I_5(\tilde t, \tilde x,\tau,z)\right) \,\mathrm{d}\tilde t\,\mathrm{d} \tilde x\,\mathrm{d}\tau\,\mathrm{d} z,
	\end{eqnarray*}
	where
	\begin{eqnarray*}
		I_1(\tilde t, \tilde x,\tau,z)&=&\modulo{\rho(\tilde t+\tau,\tilde x+z)-\tilde \rho(\tilde t-\tau,\tilde x-z)} \partial_{\tilde t} \phi\left(\tilde t,\tilde x\right),\\
		I_2(\tilde t, \tilde x,\tau,z)&=&\sign{ \rho(\tilde t+\tau,\tilde x+z)-\tilde \rho(\tilde t-\tau,\tilde x-z) }\\
		&& \times(f(\tilde t+\tau,\tilde x+z,\rho)-f(\tilde t-\tau,\tilde x-z,\tilde \rho)) \partial_{\tilde x} \phi\left(\tilde t,\tilde x\right),\\
		I_3(\tilde t, \tilde x,\tau,z)&=&-\sign{ \rho(\tilde t+\tau,\tilde x+z)-\tilde \rho(\tilde t-\tau,\tilde x-z) } \\
		&&\times\left(\partial_{\tilde x+z} f(\tilde t+\tau,\tilde x+z,\tilde \rho) -\partial_{\tilde x -z} f(\tilde t-\tau,\tilde x-z,\rho)\right)\phi\left(\tilde t,\tilde x\right)\delta_p\left(z\right) \delta_p\left(\tau \right),\\
		I_4(\tilde t, \tilde x,\tau,z)&=&-\sign{ \rho(\tilde t+\tau,\tilde x+z)-\tilde \rho(\tilde t-\tau,\tilde x-z) }\\
		&&\times \partial_{\tilde x}\phi\left(\tilde t,\tilde x\right) \delta_p\left(z\right) \delta_p\left(\tau \right)\left[ (f(\tilde t+\tau,\tilde x+z,\tilde \rho) -f(\tilde t-\tau,\tilde x-z, \tilde \rho))\right.\\ 
		&&\hspace{8em}\left.+(f(\tilde t+\tau,\tilde x+z, \rho) -f(\tilde t-\tau,\tilde x-z,\rho))\right],\\
		I_5(\tilde t, \tilde x,\tau,z)&=&\left(F(\tilde x+z,\rho(\tilde t+\tau,\tilde x+z),\tilde \rho(\tilde t-\tau,\tilde x-z))\right.\\
		&&\left.-F( \tilde x-z,\rho(\tilde t+\tau,\tilde x+z),\tilde \rho(\tilde t-\tau,\tilde x-z))\right)\phi\left(\tilde t, \tilde x\right) \partial_z\delta_p\left(z\right) \delta_p\left(\tau \right).
	\end{eqnarray*}
	Employing Lebesgue's differentiation theorem, to obtain the following limits
	\begin{eqnarray*}
		&&\ \lim_{p\to 0} {\iint\iint}_{\Omega} I_1(\tilde t, \tilde x,\tau, z) \delta_{p}(z)\delta_{p}(\tau)\, \mathrm{d}\tilde t\, \mathrm{d} \tilde x\, \mathrm{d}\tau\, \mathrm{d} z \\
		&& \hspace{7em}=\iint_{\R^+\times \R} \modulo{\rho(t,x)-\tilde \rho(t,x)} \partial_t \phi(t,x) \mathrm{d} t \mathrm{d} x,\\
		&& \lim_{p\to 0} {\iint\iint}_{\Omega} I_2(\tilde t, \tilde x, \tau, z) \delta_{p}(z)\delta_{p}(\tau)\, \mathrm{d}\tilde t\, \mathrm{d} \tilde x\, \mathrm{d}\tau\, \mathrm{d} z\\
		&& \hspace{7em}= \iint_{\R^+\times \R} \sign{ \rho(t,x)-\tilde \rho(t,x))(f(t,x,\rho)-f(t,x,\tilde \rho) } \partial_x \phi(t,x) \mathrm{d} t \mathrm{d} x.	\end{eqnarray*}
	Let us consider the term $I_3.$ Note that $I_3(\tilde t, \tilde x,\tau, z)=0$, if $\tilde x\in[-h,h],$ since then $\phi(\tilde t,\tilde x)=0$ for any $\tilde t,$ or if $\modulo{z}\geq p.$ On the other hand, if $\tilde x \not\in [-h,h],$ then $\tilde x\pm z<0$ or $\tilde x\pm z>0,$ at least when $\modulo{z}<p$ and $p<h.$
	Defining $U(t,x)=1-\omega_\eta\ast \rho$ and $V(t,x)=1-\omega_\eta\ast \tilde \rho,$  and sending $p\to0:$
	\begin{eqnarray*}
		&& \lim_{p\to 0} {\iint\iint}_{\Omega_0} I_3(\tilde t, \tilde x,  \tau, z) \, \mathrm{d}\tilde t\, \mathrm{d} \tilde x\, \mathrm{d} \tau\, \mathrm{d} z\\
		&& =\iint_{\R^+\times \R\setminus\{0\}} \sign{ \rho(t,x)-\tilde \rho(t,x) } \mathfrak{v}(x) \left(\tilde \rho g(\tilde \rho) \partial_x V- \rho g(\rho) \partial_x U\right) \phi\left(t,x\right) \, \mathrm{d} t \, \mathrm{d} x\\
		&& \leq k_r\norma{\partial_x V} \norma{g'}\iint_{\R^+\times \R\setminus\{0\}} \modulo{\rho-\tilde \rho} \phi(t,x)\, \mathrm{d} t \, \mathrm{d} x +k_r \iint_{\R^+\times \R\setminus\{0\}}\modulo{\rho g(\rho)} \modulo{\partial_x V- \partial_x U}\, \mathrm{d} t \, \mathrm{d} x\\
		&& \leq  K_1 \iint_{\R^+\times \R\setminus\{0\}} \modulo{\rho-\tilde \rho} \phi(t,x)\, \mathrm{d} t \, \mathrm{d} x, 
	\end{eqnarray*}
	where 
	\begin{eqnarray*}
		\mathfrak{v}(x)=\begin{cases}
			k_l,& \quad \hbox{if $x<0$},\\
			k_r,& \quad \hbox{if $x>0$}.
		\end{cases}
	\end{eqnarray*}		
	In fact, 
	\begin{eqnarray*}
		\modulo{\partial_{x}V-\partial_x U}&\leq& \norma{\omega'_\eta}\,  \,\norma{u(t,\cdot)-v(t,\cdot)}_{L^1}\\
		&&+\omega_\eta(0)
		\left( \modulo{u- v}(t,x+\eta)+ \modulo{u-v}(t,x) \right).
	\end{eqnarray*}
	The term $I_4$ converges to zero as $p\to0.$
	Finally, the term $I_5$ 
	\begin{equation}
	\lim_{p\to 0} {\iint\iint}_{\Omega_0} I_5(\tilde t, \tilde x, \tau, z) \, \mathrm{d} \tilde t\, \mathrm{d} \tilde x\, \mathrm{d}\tau\, \mathrm{d} z\leq  K_2 \iint_{\R^+\times \R\setminus\{0\}} \modulo{\rho-\tilde \rho} \phi(t,x)\, \mathrm{d} t \, \mathrm{d} x. 
	\end{equation}
\end{proof}

\section{Numerical simulations}\label{sec:numericaltests}
In this section, we propose some numerical tests in order to illustrate the dynamics of the non-local model (\ref{eq:CL}) with flux function discontinuous at $x=0$ and compare it with the local case. We solve the equation (\ref{eq:CL}) in an interval containing $x=0$ using the numerical scheme described in subsection (\ref{sec:numericalscheme})  for different values of $\Delta x$. For each integration, we set $\Delta t$ such that satisfies the CFL condition (\ref{cfl}), and for all tests we choose $\omega(x)=\frac{2(\eta-x)}{\eta^2}$ for $0\leq x\leq\eta$ and absorbing boundary conditions. The reference solution is computed with $\Delta x=1/1280.$

\subsection{Example 1.} We consider the initial condition $$\rho_0(x)=\begin{cases}
0.9 \quad x\in [-0.5,1.5]\\ 0.1 \quad \hbox{otherwise, } \end{cases}$$ 
$\psi(\rho)=1-\rho$, which satisfies the hypothesis \eqref{hyp_v} and $\eta=0.4$. In \textbf{Case I } we take $k_l=3$ and $k_r=1$, i.e. $v_l(\rho)>v_r(\rho)$. In Fig \ref{fig:examp1}(Left) we display the  approximated solution for $\Delta x=1/320$ at different final times $T=0.5, 1.0, 1.5, 2.0$. We can observe the formation of a stationary shock wave at $x=0$ and a queue travelling backward. We observe that the solution satisfies the maximum principle according with Lemma \ref{bounds}. 

In \textbf{Case II } we take $k_l=1$ and $k_r=3$, i.e.  $v_l(\rho)<v_r(\rho)$. In Fig \ref{fig:examp1}(Right) we display the numerical solution for $\Delta x=1/320$ at different final times $T=1.0, 2.0$. We can observe the formation of a rarefaction wave at the right of $x=0$ and the density diminishes at the left of $x=0.$ 

The $\mathbf{L}^1$-error for different $\Delta x$ at $T=2$ are computed in Table \ref{tab:example1}. 

\begin{figure}
	\centering
	\includegraphics[width=0.49\textwidth]{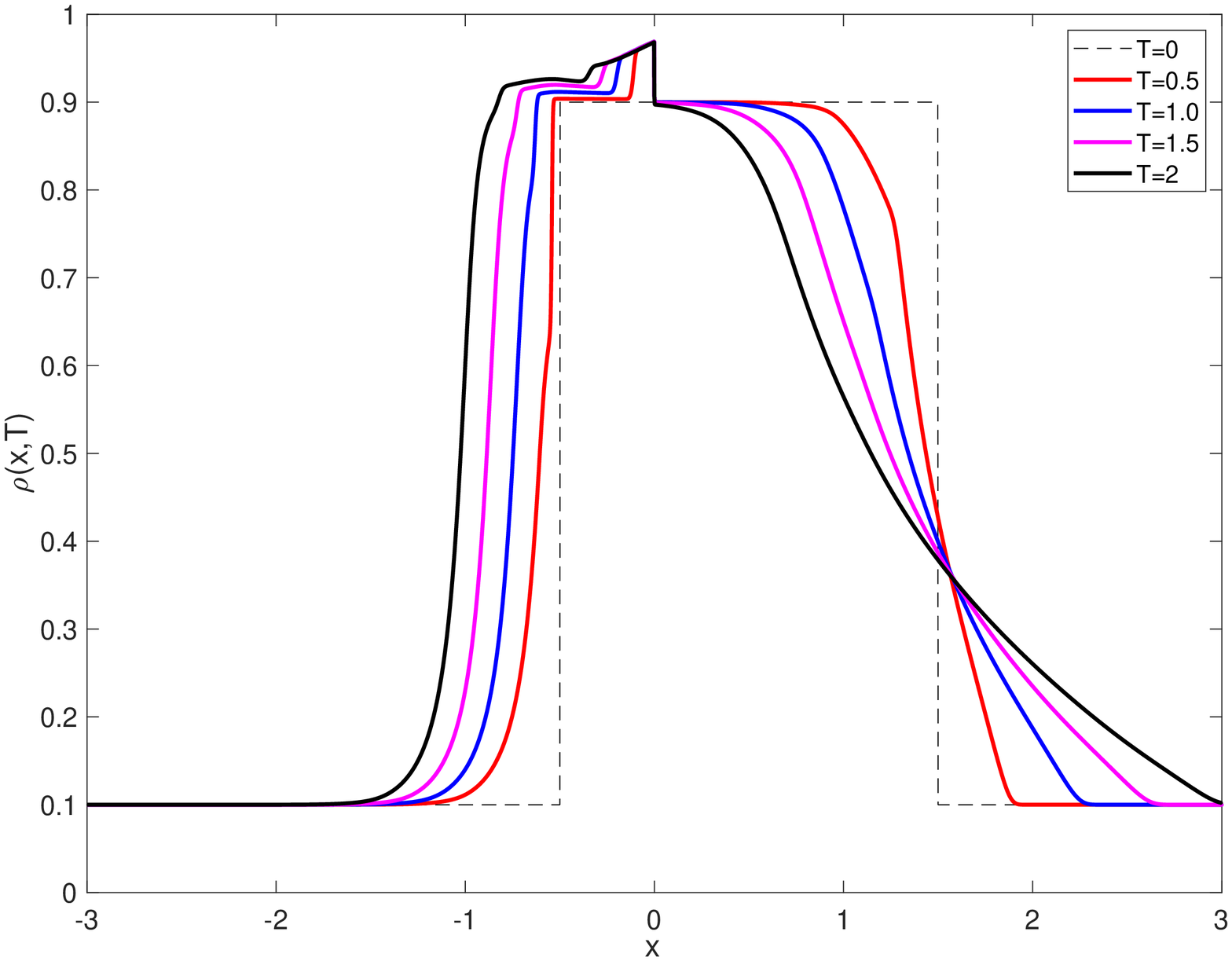} \includegraphics[width=0.49\textwidth]{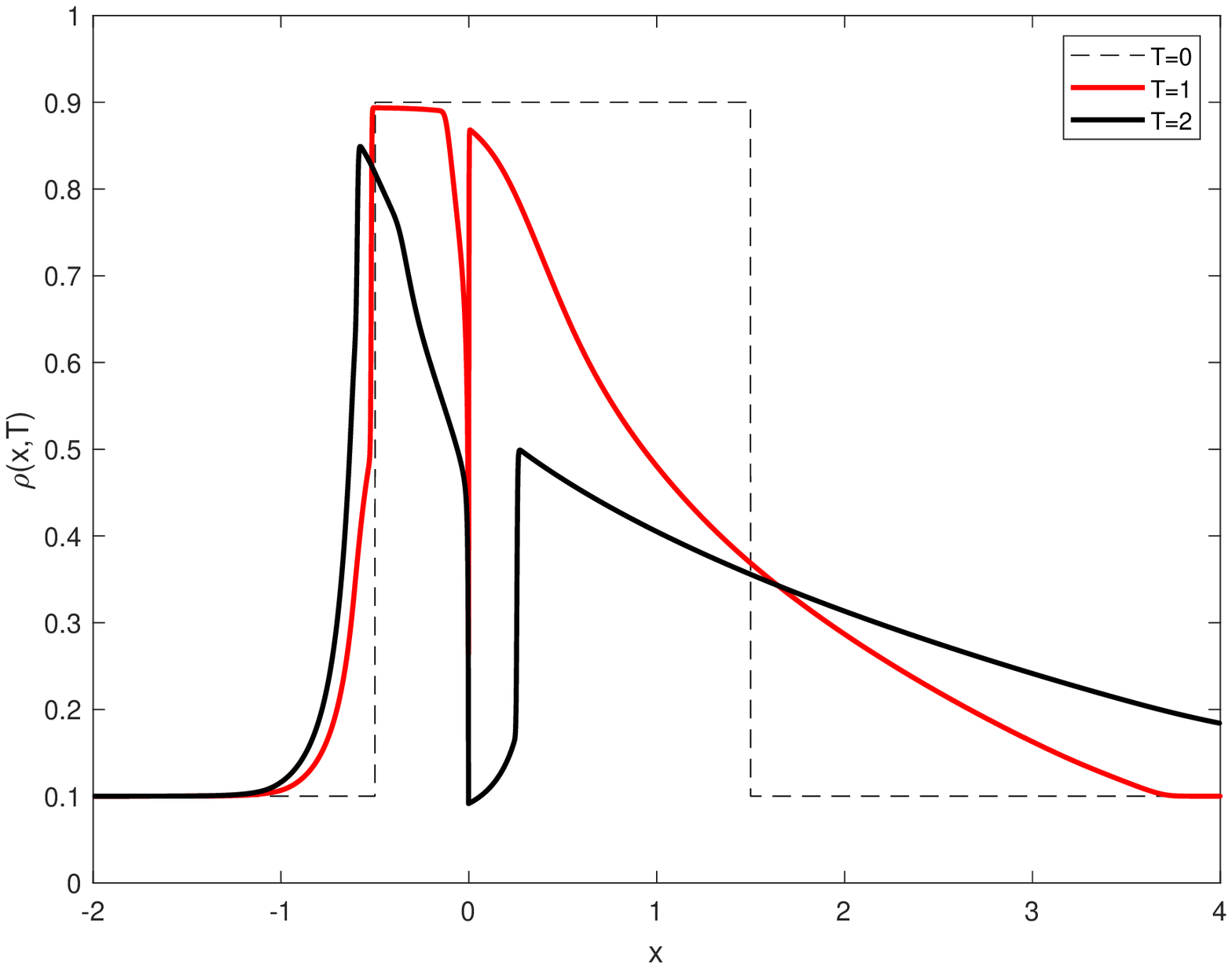}
	\caption{Example 1: Dynamics of model \eqref{eq:CL} (Left) Case $v_l(\rho)>v_r(\rho)$, (Right) Case $v_l(\rho)<v_r(\rho)$}
	\label{fig:examp1}
\end{figure}

\begin{table}[h]
	\centering
	\begin{tabular}{|c|c|c||c|c|}\hline
		&\multicolumn{2}{|c||}{Cases I} & \multicolumn{2}{|c|}{Cases II} \\\hline
		$\Delta x$ & $\mathbf{L}^1$-error & E.O.A.& $\mathbf{L}^1$-error & E.O.A.\\\hline
		$\frac{1}{40}$ & $5.7e-2$ & $-$  & $9.8e-2$& $-$\\
		$\frac{1}{80}$ & $2.8e-2$ & $1.0$ & $5.0e-2$& $1.0$\\
		$\frac{1}{160}$ & $1.4e-2$ & $1.0$ & $2.3e-2$& $1.1$ \\
		$\frac{1}{320}$ & $6.5e-3$ & $1.1$ & $1.1e-2$& $1.1$\\
		$\frac{1}{640}$& $2.4e-3$ & $1.4$ & $5.0e-3$& $1.0$ \\\hline
	\end{tabular}
	\caption{Example 1. $\mathbf{L}^1$-error and Experimental Order of Accuracy at time $T=2$.}
	\label{tab:example1}
\end{table}

\subsection{Example 2: Limit $\eta\to 0^+$.}
In this example, we investigate the numerical convergence of the approximate solution computed with the numerical scheme (\ref{scheme})-(\ref{numflux})  to the solution of the local conservation law with discontinuous flux under hypothesis (\ref{hyp_v}), as the support of the kernel function $\omega_\eta$ tends to $0^+$. In particular, we show numerical solutions at final time $T=2,$ with $\Delta x=1/1600$ and $\eta=\{0.1,0.02,0.005\}$. To evaluate the convergence, we compute the $\mathbf{L}^1$ distance between the approximate solution of the non-local problem with a given $\eta$ and the results of the classical Godunov scheme for the corresponding local problem. In Table  \ref{tab:l1distance}, we can observe than the $\mathbf{L}^1$ distance goes to zero when $\eta\to0^+$. The results are illustrated in Fig  \ref{fig:examp2}. 
\begin{figure}
	\centering
	\includegraphics[width=0.49\textwidth]{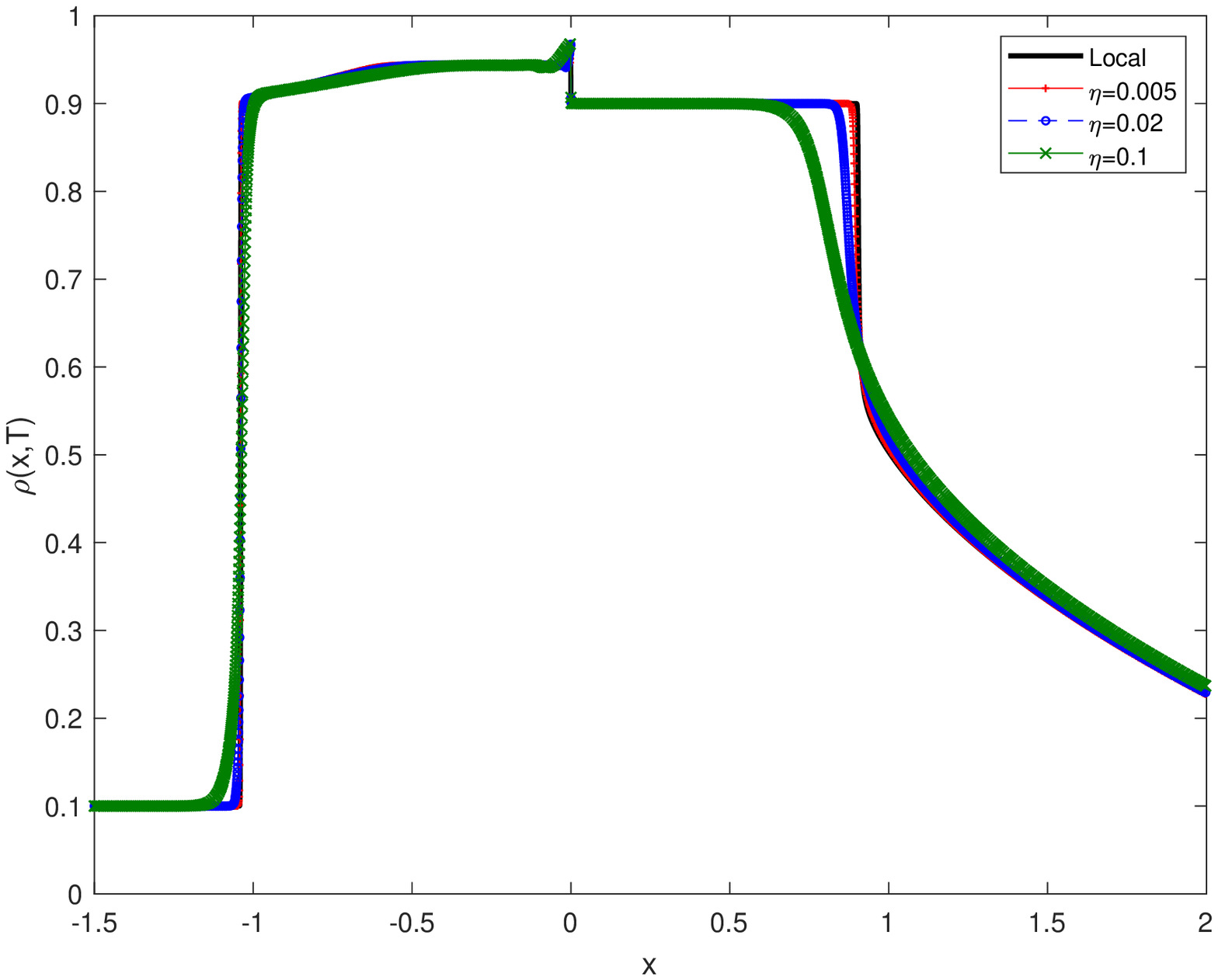} \includegraphics[width=0.49\textwidth]{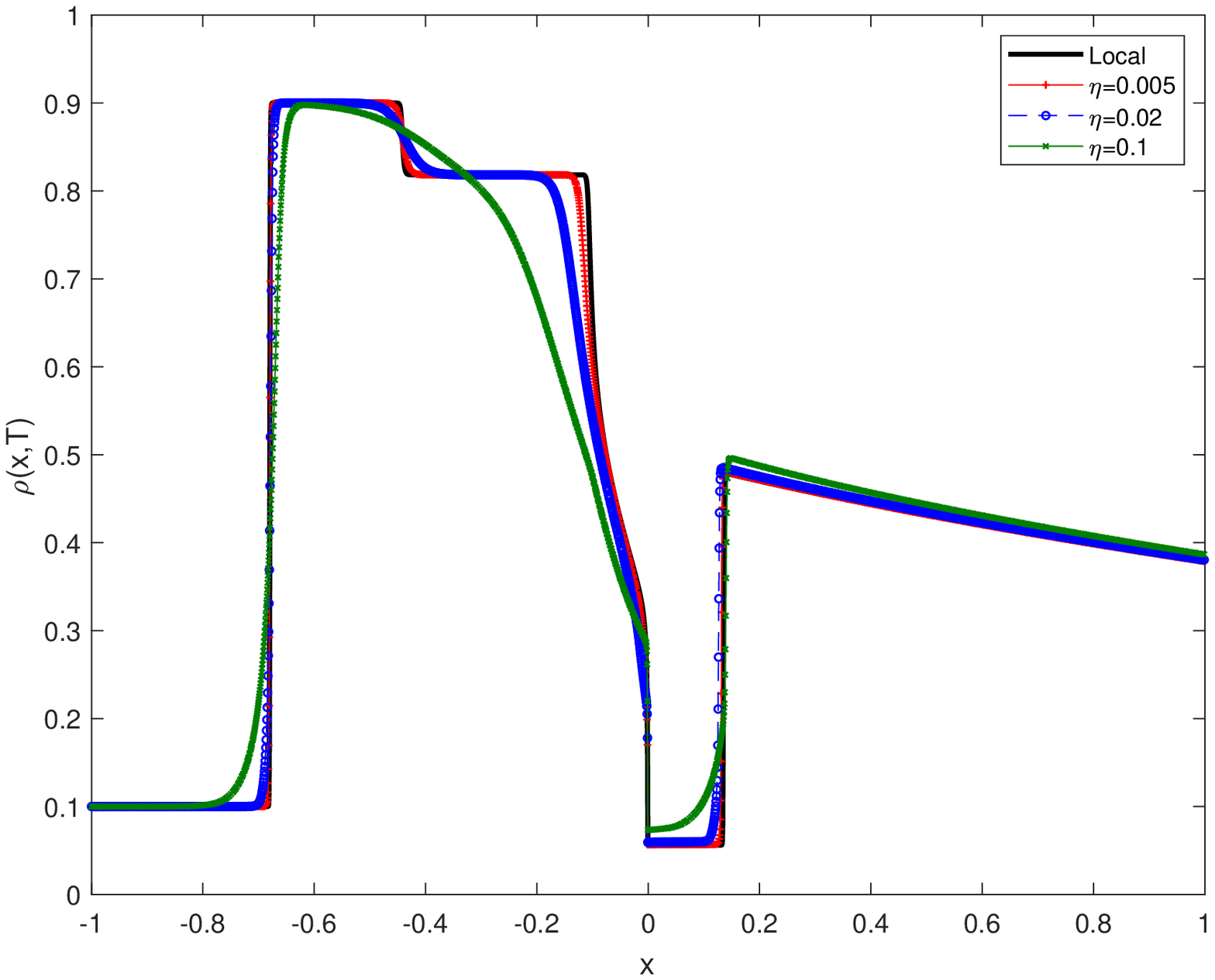} 
	\caption{Example 2. Limit $\eta\to0^+$, numerical approximations at final time $T=0.7$ with $\Delta x=1/3200$. (Left) Case I, (Right) Case II.}
	\label{fig:examp2}
\end{figure}

\begin{table}[]
	\centering
	\begin{tabular}{|c|ccc|}\hline
		& \multicolumn{3}{|c|}{$\mathbf{L}^1$ distance}\\\hline
		$\eta$    &  $0.1$ & $0.02$  & $0.005$ \\ \hline
		\textbf{Case I}  & 7.4e-2 &    2.2e-2 & 6.3e-3 \\\hline
		\textbf{Case II}  & 8.4e-2 &    2.8e-2 & 7.8e-3 \\\hline
	\end{tabular}
	\caption{Example 3. $\mathbf{L}^1$ distance between the approximate solutions to the non-local problem and the local problem for different values of $\eta$ at $T=2$ with $\Delta x=1/1600.$}
	\label{tab:l1distance}
\end{table}

\section{Conclusions and discussions}
In this paper, we have studied a non-local conservation law whose flux function is of the form $H(-x)\rho g(\rho) v_l(\omega_\eta*\rho)+H(x)\rho g(\rho) v_r(\omega_\eta*\rho)$, with a single spatial discontinuity at $x=0$ and the velocity functions satisfy the hypothesis (\ref{hyp_v}). 
We have approximated the problem through an upwind-type numerical scheme, which is a general version of the scheme proposed in \cite{2018Gottlich},
and have provided $\mathbf{L}^\infty$ and $\mathbf{BV}$ estimates for the approximate solutions. Thanks to these estimates, we have proved the well-posedness, i.e., existence and uniqueness of a weak entropy solutions.
Numerical simulations illustrate the dynamics of the studied model and corroborate the convergence of the numerical scheme. The limit model as the kernel support tends to zero is numerically investigated.




\end{document}